\documentclass[12pt]{amsart}

\newtheorem{theorem}{Theorem}[section]
\newtheorem{lemma}[theorem]{Lemma}
\newtheorem{corollary}[theorem]{Corollary}

\theoremstyle{definition}
\newtheorem{definition}[theorem]{Definition}
\newtheorem{example}[theorem]{Example}

\theoremstyle{remark}
\newtheorem{remark}[theorem]{Remark}

\let\olddefin\definition
\renewcommand{\definition}{\olddefin\normalfont}
\let\oldlemma\lemma
\renewcommand{\lemma}{\oldlemma\normalfont}
\let\oldtheorem\theorem
\renewcommand{\theorem}{\oldtheorem\normalfont}
\let\oldcoro\corollary
\renewcommand{\corollary}{\oldcoro\normalfont}

\let\oldremark\remark
\renewcommand{\remark}{\oldremark\normalfont}
\let\oldexample\example
\renewcommand{\example}{\oldexample\normalfont}

\numberwithin{equation}{section}

\usepackage{amscd}
\usepackage{amsmath}

\begin{document}

\title{Higher algebraic K-theory and tangent spaces to Chow groups}

\author{Sen Yang}
\address{Mathematics Department\\
Louisiana State University\\
Baton Rouge, Louisiana}
\email{senyangmath@gmail.com}

\maketitle

\begin{abstract}
This is a reproduction of my thesis [35]. By using higher K-theory(Chern character, cyclic homology, effacement theorem and etc), we provide an answer to a question by Green-Griffiths which says the tangent sequence to Bloch-Gersten-Quillen sequence is Cousin resolution.
\end{abstract}

\tableofcontents

\section{Introduction}
\label{Introduction}
In this paper, using higher algebraic K-theory, we provide an answer to the following question asked by Green-Griffiths in [10](more details are discussed in section 2):

\begin{quote}
Can one define the Bloch-Gersten-Quillen sequence $\mathcal{G}_j$ on 
infinitesimal neighborhoods $X_j = X \times Spec(k[t]/(t^{j+1})$
so that 
\[
 ker(\mathcal{G}_1 \to \mathcal{G}_0) =  \underline{\underline{T}}\mathcal{G}_0,
\]
\end{quote}
here $ \underline{\underline{T}}\mathcal{G}_0$ should be the 
Cousin resolution of $TK_{m}(\mathcal{O}_X)$ and $X$ is any $n$-dimensional smooth projective variety over  a field $k$, $chark=0$.

Our main results are as follows.

\begin{definition}
definition-theorem[7]

Let $T_{j}$ denote $Spec(k[t]/(t^{j+1})$ and $X_{j}$ denote the $j$-th infinitesimal thickening, $X \times T_{j}$, the Bloch-Quillen-Gersten sequence $\mathcal{G}_j$ on $X_{j}$ is defined to be the following flasque resolution($m$ can be any integer):
{\footnotesize
\begin{align*}
0 \to & K_{m}(O_{X_{j}}) \to K_{m}(k(X)_{j}) \to \bigoplus_{x_{j} \in X_{j} ^{(1)}}\underline{K}_{m-1}(O_{X_{j},x_{j}} \ on \ x_{j}) \to \dots \\
 & \dots \to \bigoplus_{x_{j} \in X_{j} ^{(n)}}\underline{K}_{m-n}(O_{X_{j},x_{j}} \ on \ x_{j}) \to 0.
\end{align*}
}
where $O_{X_{j}}=O_{X\times T_{j}}$,  $k(X)_{j}= k(X)\times T_{j}$, $x_{j}=x \times T_{j}$. $\underline{K}_{p}(O_{X_{j},x_{j}} \ on \ x_{j})$ is the flasque sheaf $j_{x_{j}\ast}K_{p}(O_{X_{j},x_{j}} \ on \ x_{j})$, where 
$j_{x_{j}}$ is the immersion $\{x_{j}\} \to X$.

\end{definition}

\begin{remark}
The existence of the sequence $\mathcal{G}_j$ and the fact of flasque resolution follows from [7]. We refer to theorem 3.6 on page 8 and corollary 3.12 on page 14.
\end{remark}

The answer to Green-Griffiths' question is positive. The following theorems are proved in section 5.
\begin{theorem}
 The formal tangent sequence to the Bloch-Gersten-Quillen 
sequence is the Cousin resolution. That is there exists the following commutative diagram(each column is a flasque resolution, $m$ can be any integer, $\varepsilon$ is a nilpotent with $\varepsilon^2=0$. ):
{\footnotesize
\[
  \begin{CD}
     0 @. 0 @. 0\\
     @VVV @VVV @VVV\\
     \Omega_{O_{X}/ \mathbb{Q}}^{\bullet} @<tan1<< K_{m}(O_{X[\varepsilon]}) @<<< K_{m}(O_{X}) \\
     @VVV @VVV @VVV\\
     \Omega_{k(X)/ \mathbb{Q}}^{\bullet} @<tan2<<  K_{m}(k(X)[\varepsilon]) @<<< K_{m}(k(X)) \\
     @VVV @VVV @VVV\\
     \oplus_{d \in X^{(1)}}\underline{H}_{d}^{1}(\Omega_{O_{X}/\mathbb{Q}}^{\bullet}) @<tan3<< \oplus_{d[\varepsilon]\in X[\varepsilon]^{(1)}}\underline{K}_{m-1}(O_{X,d}[\varepsilon] \ on \ d[\varepsilon]) @<<<  \oplus_{d \in X^{(1)}}\underline{K}_{m-1}(O_{X,d} \ on \ d)\\
     @VVV @VVV @VVV\\
     \oplus_{y \in X^{(2)}}\underline{H}_{y}^{2}(\Omega_{O_{X}/ \mathbb{Q}}^{\bullet}) @<tan4<< \oplus_{y[\varepsilon] \in X[\varepsilon]^{(2)}}\underline{K}_{m-2}(O_{X,y}[\varepsilon] \ on \ y[\varepsilon]) @<<< \oplus_{y \in X^{(2)}}\underline{K}_{m-2}(O_{X,y} \ on \ y) \\
     @VVV @VVV @VVV\\
      \dots @<tan<< \dots @<<< \dots \\ 
     @VVV @VVV @VVV\\
     \oplus_{x\in X^{(n)}}\underline{H}_{x}^{n}(\Omega_{O_{X}/ \mathbb{Q}}^{\bullet}) @<tan(n+2)<< \oplus_{x[\varepsilon]\in X[\varepsilon]^{(n)}}\underline{K}_{m-n}(O_{X,x}[\varepsilon] \ on \ x[\varepsilon]) @<<<  \oplus_{x \in X^{(n)}}\underline{K}_{m-n}(O_{X,x} \ on \ x) \\
     @VVV @VVV @VVV\\
      0 @. 0 @. 0
  \end{CD}
\]
}
where
\begin{equation}
\begin{cases}
 \begin{CD}
 \Omega_{O_{X}/ \mathbb{Q}}^{\bullet} = \Omega^{m-1}_{O_{X}/\mathbb{Q}}\oplus \Omega^{m-3}_{O_{X}/\mathbb{Q}}\oplus \dots \\
 \Omega_{k(X)/ \mathbb{Q}}^{\bullet} = \Omega^{m-1}_{k(X)/\mathbb{Q}}\oplus \Omega^{m-3}_{k(X)/\mathbb{Q}}\oplus \dots
\end{CD}
\end{cases}
\end{equation}

\end{theorem}

Adams operations $\psi^{k}$ can decompose the above diagram into eigen-components. To be precise, let $K^{(i)}_{m}(O_{X,x} \ on \ x)$ denote the eigen-space of $\psi^{k} = k^{i}$, we have the following refiner result.
\begin{theorem}
 There exists the following  commutative diagram((each column is a flasque resolution, $m$ can be any integer.):
{\footnotesize
\[
  \begin{CD}
     0 @. 0 @. 0\\
     @VVV @VVV @VVV\\
     \Omega_{O_{X}/ \mathbb{Q}}^{\bullet,(i)} @<tan1<< K^{(i)}_{m}(O_{X[\varepsilon]}) @<<< K^{(i)}_{m}(O_{X}) \\
     @VVV @VVV @VVV\\
     \Omega_{k(X)/ \mathbb{Q}}^{\bullet,(i)} @<tan2<<  K^{(i)}_{m}(k(X)[\varepsilon]) @<<< K^{(i)}_{m}(k(X)) \\
     @VVV @VVV @VVV\\
     \oplus_{d \in X^{(1)}}\underline{H}_{d}^{1}(\Omega_{O_{X}/\mathbb{Q}}^{\bullet,(i)}) @<tan3<< \oplus_{d[\varepsilon]\in X[\varepsilon]^{(1)}}\underline{K}^{(i)}_{m-1}(O_{X,d}[\varepsilon] \ on \ d[\varepsilon]) @<<<  \oplus_{d \in X^{(1)}}\underline{K}^{(i)}_{m-1}(O_{X,d} \ on \ d)\\
     @VVV @VVV @VVV\\
     \oplus_{y \in X^{(2)}}\underline{H}_{y}^{2}(\Omega_{O_{X}/ \mathbb{Q}}^{\bullet,(i)}) @<tan4<< \oplus_{y[\varepsilon] \in X[\varepsilon]^{(2)}}\underline{K}^{(i)}_{m-2}(O_{X,y}[\varepsilon] \ on \ y[\varepsilon]) @<<< \oplus_{y \in X^{(2)}}\underline{K}^{(i)}_{m-2}(O_{X,y} \ on \ y) \\
     @VVV @VVV @VVV\\
      \dots @<tan<< \dots @<<< \dots \\ 
     @VVV @VVV @VVV\\
     \oplus_{x\in X^{(n)}}\underline{H}_{x}^{n}(\Omega_{O_{X}/ \mathbb{Q}}^{\bullet,(i)}) @<tan(n+2)<< \oplus_{x[\varepsilon]\in X[\varepsilon]^{(n)}}\underline{K}^{(i)}_{m-n}(O_{X,x}[\varepsilon] \ on \ x[\varepsilon]) @<<<  \oplus_{x \in X^{(n)}}\underline{K}^{(i)}_{m-n}(O_{X,x} \ on \ x) \\
     @VVV @VVV @VVV\\
      0 @. 0 @. 0
  \end{CD}
\]
}
where 
\begin{equation}
\begin{cases}
 \begin{CD}
 \Omega_{O_{X}/ \mathbb{Q}}^{\bullet,(i)}= \Omega^{{2i-m+1}}_{O_{X}/ \mathbb{Q}}, for \  \frac{m-1}{2}  < \ i \leq m-1.\\
  \Omega_{O_{X}/ \mathbb{Q}}^{\bullet,(i)}= 0, else.
 \end{CD}
\end{cases}
\end{equation} 
\end{theorem}

Moreover, we have 
\begin{theorem}
There exists the following commutative diagram(each column is a flasque resolution, $m$ can be any integer.):
{\footnotesize
\[
  \begin{CD}
     0 @. 0 @. 0\\
     @VVV @VVV @VVV\\
     (\Omega_{O_{X}/ \mathbb{Q}}^{\bullet})^{\oplus j} @<tan1<< K_{m}(O_{X_{j}}) @<<< K_{m}(O_{X}) \\
     @VVV @VVV @VVV\\
     (\Omega_{k(X)/ \mathbb{Q}}^{\bullet})^{\oplus j} @<tan2<<  K_{m}(k(X)_{j}) @<<< K_{m}(k(X)) \\
     @VVV @VVV @VVV\\
     \oplus_{d \in X^{(1)}}\underline{H}_{d}^{1}((\Omega_{O_{X}/\mathbb{Q}}^{\bullet})^{\oplus j} @<tan3<< \oplus_{d_{j} \in X_{j} ^{(1)}}\underline{K}_{m-1}(O_{X_{j},d_{j}} \ on \ d_{j}) @<<<  \oplus_{d \in X^{(1)}}\underline{K}_{m-1}(O_{X,d} \ on \ d)\\
     @VVV @VVV @VVV\\
     \oplus_{y \in X^{(2)}}\underline{H}_{y}^{2}((\Omega_{O_{X}/ \mathbb{Q}}^{\bullet})^{\oplus j} @<tan4<< \oplus_{y_{j} \in X_{j} ^{(2)}}\underline{K}_{m-2}(O_{X_{j},y_{j}} \ on \ y_{j}) @<<< \oplus_{y \in X^{(2)}}\underline{K}_{m-2}(O_{X,y} \ on \ y) \\
     @VVV @VVV @VVV\\
      \dots @<tan<< \dots @<<< \dots \\ 
     @VVV @VVV @VVV\\
     \oplus_{x\in X^{(n)}}\underline{H}_{x}^{n}((\Omega_{O_{X}/ \mathbb{Q}}^{\bullet})^{\oplus j} @<tan(n+2)<< \oplus_{x_{j} \in X_{j} ^{(n)}}\underline{K}_{m-n}(O_{X_{j},x_{j}} \ on \ x_{j}) @<<<  \oplus_{x \in X^{(n)}}\underline{K}_{m-n}(O_{X,x} \ on \ x) \\
     @VVV @VVV @VVV\\
      0 @. 0 @. 0
  \end{CD}
\]
}
where
\begin{equation}
\begin{cases}
 \begin{CD}
 \Omega_{O_{X}/ \mathbb{Q}}^{\bullet} = \Omega^{m-1}_{O_{X}/\mathbb{Q}}\oplus \Omega^{m-3}_{O_{X}/\mathbb{Q}}\oplus \dots \\
 \Omega_{k(X)/ \mathbb{Q}}^{\bullet} = \Omega^{m-1}_{k(X)/\mathbb{Q}}\oplus \Omega^{m-3}_{k(X)/\mathbb{Q}}\oplus \dots
\end{CD}
\end{cases}
\end{equation} 

\end{theorem}

 The proof for the above theorem, given in section 5, requires non-trivial techniques from higher algebraic K-theory and negative cyclic homology.
 The main points are outlined as follows.
 \begin{itemize}
  \item  According to [5],  there exists a  Chern character from K-theory spectrum $\mathcal{K}$ to negative cyclic homology spectrum $\mathcal{HN}$,
\[
 Ch: \mathcal{K} \to \mathcal{HN},
\]
here $\mathcal{K}$ is the Thomason-Trobaugh spectrum and $\mathcal{HN}$ is the spectrum associated to negative cyclic complex constructed by Keller[18,20].
This Chern character induces maps from the coniveau spectral sequence associated to $\mathcal{K}$ to the coniveau spectral sequence associated to $\mathcal{HN}$.

\item Effacement  theorem.

In our approach, $\mathcal{K}$ and $\mathcal{HN}$ are considered as ``cohomology theories with support'' in the sense of [7].
Both $\mathcal{K}$ and $\mathcal{HN}$ satisfy \'etale excision and projective bundle formula. Therefore, according to [7], $\mathcal{K}$ and $\mathcal{HN}$ 
are effaceable functors. Effacement theorem gives us the exactness and universal exactness of sheafified Bloch-Gersten-Quillen sequence.

\item Goodwillie-type  and  Cathelineau-type   results

Goodwillie -type and Cathelineau-type results enable us to compute relative K-groups(with support) in terms of relative negative cyclic groups(with support). Our computation is based on a recent version
by Corti\~nas-Haesemeyer-Weibel[6].
 \end{itemize}

This paper is organized as follows.  We begin with an introduction of Green-Griffiths' work and their question in section 2. 

In section 3, we discuss effacement theorem and Chern character which are the first two ingredients for proving our main result.

Lambda and Adams operations are discussed in section 4. We show Goodwillie-type and Cathelineau-type results, which are the third ingredient for proving our main results which are proved in section 5.
\section*{Acknowledgements}
This note is a reproduction of my thesis. I would like to express my deep gratitude to professor J.W.Hoffman for being a wonderful advisor. I also want to sincerely thank professor M.Schlichting for teaching me higher K-theory and for inviting me to visit Warwick Mathematics Institute. Many thanks to K-theory experts P.Balmer, C.Soul\'e and C.Weibel. 

I am very grateful to professor M.Green and professor P.Griffiths for asking interesting questions and for enlightening discussions. Their questions, talks and papers have completely reformulated my understanding of algebraic cycles.

It is a pleasure to thank professor James Oxley for various help during the last 5 years.

Last not least, I want to sincerely thank Department of Mathematics of LSU for financial support and for providing me with a pleasant working environment.

\section{Green-Griffiths' question}
\label{Green-Griffiths' question}
 %Beginning with Bloch, Gersten and Quillen, $K$-theory enters into the picture of studying higher codimensional algebraic cycles. 
 
 Green-Griffiths initiated the study of tangent spaces to Chow groups in [10]. In the following, we use points on a surface, $CH^{2}(X)$, to explain their ideas. 
 
The well-known Bloch-Gersten-Quillen flasque resolution
\[
    0 \to K_{2}(O_{X}) \to K_{2}(\mathbb{C}(X)) \to \bigoplus_{y \in X^{(1)}}\underline{K}_{1}( \mathbb{C}(y)) \to \bigoplus_{x\in X^{(2)}} \underline{K}_{0}( \mathbb{C}(x)) \to 0 
\]
leads to the Bloch's formula:
\[
 CH^{2}(X)=H^{2}(X,K_{2}(O_{X})).
\]

Recall that the tangent space to $K_{2}(O_{X})$ is defined to be $TK_{2}(O_{X})= Ker\{K_{2}(O_{X}[\varepsilon]) \xrightarrow{\varepsilon = 0} K_{2}(O_{X}) \}$. Combining with Van der Kallen's isomorphism which identifies $TK_{2}(O_{X})$ with absolute differentials
  \[
     TK_{2}(O_{X})=\Omega_{X/ \mathbb{Q}}^{1},
  \]
one can formally define tangent space to Chow group to be  
 \[
     TCH^{2}(X)= H^{2}(X,TK_{2}(O_{X}))=H^{2}(X,\Omega_{X/ \mathbb{Q}}^{1}).
  \]
  
Green and Griffiths would like to understand the geometric significance of the above isomorphism. The clue to do this is from the Cousin flasque resolution
\[ 
    0 \to \Omega_{X/ \mathbb{Q}}^{1} \to \Omega_{\mathbb{C}(X)/ \mathbb{Q}}^{1} \to \bigoplus_{y \in X^{(1)}}H_{y}^{1}(\Omega_{X/ \mathbb{Q}}^{1}) \xrightarrow{\partial}  \bigoplus_{x\in X^{(2)}} H_{x}^{2}(\Omega_{X/ \mathbb{Q}}^{1}) \to 0 
\]
which gives rise to 
 \[
   H^{2}(X,\Omega_{X/ \mathbb{Q}}^{1})=\dfrac{\bigoplus_{x\in X^{(2)}}H_{x}^{2}(\Omega_{X/ \mathbb{Q}}^{1})}{Im(\partial)}.
 \]
 
Following the above information and the classical definition$CH^{2}(X)=\dfrac{Z^{2}(X)}{Z^{2}_{rat}(X)}$, Green and Griffiths define the tangent space $TZ^{2}(X)$
to the 0-cycles on $X$ as 
\[
 TZ^{2}(X)=\bigoplus_{x\in X^{(2)}} H_{x}^{2}(\Omega_{X/ \mathbb{Q}}^{1}),
\]
and define a tangent subspace $TZ^{2}_{rat}(X)$ to the rational equivalence as 
\[
 TZ^{2}_{rat}(X)=Im(\partial).
\]

The question is to show that there is really a tangent map from cycles to the local cohomology. In other words, Green and Griffiths ask for a Bloch-Gersten-Quillen
 type exact sequence which can fill in the middle in the following diagram
\[\displaystyle
  \begin{CD}
     0 @. 0 @. 0\\
     @VVV @VVV @VVV\\
     \Omega_{X/ \mathbb{Q}}^{1} @<tan1<< K_{2}(X[\varepsilon]) @>\varepsilon=0>> K_{2}(X) \\
     @VVV @VVV @VVV\\
     \Omega_{\mathbb{C}(X)/ \mathbb{Q}}^{1} @<tan2<<  K_{2}(\mathbb{C}(X)[\varepsilon]) @>\varepsilon=0>> K_{2}(\mathbb{C}(X)) \\
     @VVV @VVV @VVV\\
     \bigoplus_{y\in X^{(1)}}H_{y}^{1}(\Omega_{X/\mathbb{Q}}^{1}) @<tan3<<  Arcs^{1}(X) @>\varepsilon=0>> \bigoplus_{y\in X^{(1)}}K_{1}(\mathbb{C}(y)) \\
     @VVV @VVV @VVV\\
     \bigoplus_{x\in X^{(2)}} H_{x}^{2}(\Omega_{X/ \mathbb{Q}}^{1}) @<tan4<<  Arcs^{2}(X) @>\varepsilon=0>> \bigoplus_{x\in X^{(2)}}K_{0}(\mathbb{C}(x)) \\
     @VVV @VVV @VVV\\
      0 @. 0 @. 0
  \end{CD}
\]
where $Arcs^{1}(X)$ and  $Arcs^{2}(X)$ stand for the arc space associated with $\bigoplus_{y\in X^{(1)}}K_{1}(\mathbb{C}(y))$ and $\bigoplus_{x\in X^{(2)}}K_{0}(\mathbb{C}(x))$ respectively.

Green and Griffiths implicitly introduce groups of ``Arcs". The idea is that, an element of $Arcs^{1}(X)$ should be a formal sum of the expression(a pair) $\{div(f+\varepsilon f_{1}),g+\varepsilon g_{1}\}$,
where $f=0$ is a local expression for a divisor on $X$ and $g \in \mathbb{C}(Y)^{\ast}$. We think of $f+\varepsilon f_{1}$ as the $1^{st}$ order deformation of $Y=div(f)$
and $g+\varepsilon g_{1}$ is a deformation of $g$. The tangent to the $\{div(f+\varepsilon f_{1}),g+\varepsilon g_{1}\}$ is defined in the following way:

First, the following diagram
\begin{equation}
\begin{cases}
 \begin{CD}
   O_{X,y} @>f>> O_{X,y} @>>> O_{X,y}/(f) @>>> 0  \\
   O_{X,y} @>\frac{g_{1} df}{g}-\frac{f_{1} dg}{g}>> \Omega_{X/ \mathbb{Q},y}^{1}
 \end{CD}
\end{cases}
\end{equation} 
gives an element $\alpha$ in $Ext_{O_{X,y}}^{1}(O_{X,y}/(f),\Omega_{X/ \mathbb{Q},y}^{1})$. Noting that 
\[
H_{y}^{1}(\Omega_{X/\mathbb{Q}}^{1})=\varinjlim_{n \to \infty}Ext_{O_{X,y}}^{1}(O_{X,y}/(f)^{n},\Omega_{X/ \mathbb{Q},y}^{1}),
\]
the image $[\alpha]$ of $\alpha$ under the limit is in $H_{y}^{1}(\Omega_{X/\mathbb{Q}}^{1})$ and it is the tangent to $\{div(f+\varepsilon f_{1}),g+\varepsilon g_{1}\}$.

Similarly, an element of $Arcs^{2}(X)$ is a first order deformation of 0-cycles on $X$. To be precise, it is of the form $Z_{\varepsilon}=V(u+\varepsilon u_{1},v+\varepsilon v_{1})$, where
$Z=V(u,v)$ is supported on $x$ and $V(-)$ means zero set. The tangent to $Z_{\varepsilon}$ is defined in the following way. First, the following diagram
\begin{equation}
\begin{cases}
 \begin{CD}
  O_{X,x} @>(v,-u)>> O_{X,x}^{\oplus 2} @>(u,v)>>  O_{X,x} @>>> O_{X,x}/(u,v) @>>> 0\\
  O_{X,x} @>v_{1} du-u_{1} dv>> \Omega_{X/ \mathbb{Q},x}^{1}
 \end{CD}
\end{cases}
\end{equation}
gives an element $\beta$ in $Ext_{O_{X,x}}^{2}(O_{X,x}/(u,v),\Omega_{X/ \mathbb{Q},x}^{1})$. Noting that 
\[
H_{x}^{2}(\Omega_{X/\mathbb{Q}}^{1})=\varinjlim_{n \to \infty}Ext_{O_{X,x}}^{2}(O_{X,x}/(u,v)^{n},\Omega_{X/ \mathbb{Q},x}^{1}),
\]
the image $[\beta]$ of $\beta$ under the limit is in $H_{x}^{2}(\Omega_{X/\mathbb{Q}}^{1})$ and it is the tangent to $Z_{\varepsilon}=V(u+\varepsilon u_{1},v+\varepsilon v_{1})$.

By some heuristic arguments, Green and Griffiths conclude that 

\begin{theorem}

The tangent sequence to the Bloch-Gersten-Quillen sequence 
\[
    0 \to K_{2}(X) \to K_{2}(\mathbb{C}(X)) \to \bigoplus_{y \in X^{(1)}}K_{1}( \mathbb{C}(y)) \to \bigoplus_{x\in X^{(2)}} K_{0}( \mathbb{C}(x)) \to 0 
\]
is the Cousin flasque resolution of $\Omega_{X/ \mathbb{Q}}^{1}$
\[ 
    0 \to \Omega_{X/ \mathbb{Q}}^{1} \to \Omega_{\mathbb{C}(X)/ \mathbb{Q}}^{1} \to \bigoplus_{y \in X^{(1)}}H_{y}^{1}(\Omega_{X/ \mathbb{Q}}^{1}) \to  \bigoplus_{x\in X^{(2)}} H_{x}^{2}(\Omega_{X/ \mathbb{Q}}^{1}) \to 0.
\]
 
\end{theorem}

More generally, Green and Griffths pose the following question:
\begin{quote}
Can one define the Bloch-Quillen-Gersten sequence $\mathcal{G}_j$ on 
infinitesimal neighborhoods $X_j = X \times Spec(k[t]/(t ^{j+1})$
so that 
\[
 ker(\mathcal{G}_1 \to \mathcal{G}_0) =  \underline{\underline{T}}\mathcal{G}_0
\]
\end{quote}
What's meant here is that $ \underline{\underline{T}}\mathcal{G}_0$ should be the 
Cousin resolution of $TK_m (\mathcal{O}_X)$. In more
detail, 
if $X$ is smooth of dimension $n$ over a field $k$ of characteristic $0$, and if we
denote by 
$\mathcal{G}_0$ the Bloch-Quillen-Gersten resolution($m$ can be any integer) 

{\footnotesize
\begin{align*}
0 \to & K_{m}(O_{X}) \to K_{m}(k(X)) \to \bigoplus_{d \in X ^{(1)}}\underline{K}_{m-1}(O_{X,d} \ on \ d) \to \dots \\
 & \dots \to \bigoplus_{x \in X ^{(n)}}\underline{K}_{m-n}(O_{X,x} \ on \ x) \to 0.
\end{align*}
}
They are asking for analogs of this with $X$ replaced by infinitesimal 
thickenings $X_j$ in such a way that 
$ker(\mathcal{G}_1 \to \mathcal{G}_0)$ is the Cousin complex of 
$TK_m (\mathcal{O}_X)$.

We will provide an answer to this question in the following sections. 

\begin{remark}
In our setting, $Arcs^{1}(X)$ and $Arcs^{2}(X)$ are defined as(in theorem 5.6, $m=n=2$.) 
\[
 Arcs^{1}(X) = \bigoplus_{y[\varepsilon]\in X[\varepsilon]^{(1)}}K_{1}(O_{X,y}[\varepsilon] \ on \ y[\varepsilon]),
\]
\[
 Arcs^{2}(X) = \bigoplus_{x[\varepsilon]\in X[\varepsilon]^{(2)}}K_{0}(O_{X,x}[\varepsilon]  \ on \ x[\varepsilon]).
\]

On the other hand, our tangent maps are induced by Chern character, corollary 5.5. It is a quite interesting question to compare our tangent maps(induced by Chern character) with Green-Griffiths'.  
\end{remark}

\section{Effacement theorem and Chern character}
\label{Effacement theorem and Chern character}

In this section, we discuss effacement theorem and Chern character which are the first two ingredients for proving our main result. 

In section 3.1, we discuss effeacement theorem,
mainly following [7] and [28]. That is, we consider $\mathcal{K}$ and $\mathcal{HN}$ as ``cohomology theories with support'' in the sense of [8]. Both $\mathcal{K}$ and $\mathcal{HN}$
are effaceable functors, since they satisfy \'etale excision and projective bundle formula. 

In section 3.2, we recall the existence of Chern character from K-theory spectrum $\mathcal{K}$ to negative cyclic homology 
spectrum $\mathcal{HN}$, following [5].  This Chern character induces maps from the coniveau spectral sequence associated to $\mathcal{K}$ to  the coniveau spectral sequence associated to $\mathcal{HN}$. 
\subsection{Effacement theorem}
\label{effacement theorem}
Now, we would like to discuss the effacement theorem which  enables one to prove the Bloch-Ogus theorem. The following background is from [7]. The interested readers can check more  detail of the effacement theorem in [7]. 

As an important result in algebraic geometry,
the Bloch-Ogus theorem, briefly described, is as follows. Given a smooth algebraic variety $X$ and a cohomology theory $h^{\ast}$ satisfying \'etale excision
and ``Key lemma", filtration by codimension of support yields Cousin complexes which form the $E_{1}$-terms of the coniveau
spectral sequence converging to $h^{\ast}(X)$. Restriction of the Cousin complexes to the open subsets of $X$ defines complexes of flasque Zariski sheaves.
The Bloch-Ogus theorem says that these complexes of sheaves are acyclic, except in degree $0$ where their cohomology is the Zariski sheaf 
$\mathcal{H^{\ast}}$ associated to the presheaf $U \to h^{\ast}(U)$. This identifies the $E_{2}$-term of the coniveau spectral sequence to $H^{\ast}(X,\mathcal{H^{\ast}})$.

Bloch-Ogus reduce their theorem to proving the ''effacement theorem"(see below for precise statement) which is proved by using a geometric presentation lemma. Later, 
Gabber gave a different proof of effacement theorem for \'etale cohomology by essentially using the section at infinity (coming from an embedding of the affine line into the 
projective line) as well as a computation of the cohomology of the projective line.

Gabber's proof gives us more. In [7], Colliot-Th\'el\`ene, Hoobler and Kahn axiomatize Gabber's argument and show that Gabber's argument applies to any ``Cohomology theory with support'' which satisfies
\'etale excision and ``Key lemma". The latter follows either from homotopy invariance or from projective bundle formula. For a list of such cohomology theory with support,
we refer the readers to [7].

Now we adopt Colliot-Th\'el\`ene, Hoobler and Kahn's discussion to our setting. The following expression is essentially following the lecture given by M.Schlichting[28].

Let's begin by defining a ``cohomology theory with support" to a pair $(X,Z)$, where $Z$ is closed in a scheme $X$.
\begin{definition}[28]
 Let $\mathcal{A}$ be a functor from the category $Sch^{op}/k$ to
spectra or chain complexes:(The readers can take $\mathcal{A}$ to be K-theory for spectra and (negative)cyclic homology for complexes)
\[
 \mathcal{A} : Sch^{op}/k \to spectra
\]
or
\[
 \mathcal{A} : Sch^{op}/k \to chain \  complexes
\]
then we can extend $\mathcal{A}$ to a pair $(X,Z)$, where $Z$ is closed in $X$ as follows. 

For $\mathcal{A}$ spectrum-valued,
$\mathcal{A}(X \ on \ Z)$ is defined as the homotopy fiber of $\mathcal{A}(X)\to \mathcal{A}(X - Z)$ 
\[
 \mathcal{A}(X \ on \ Z) \to \mathcal{A}(X)\to \mathcal{A}(X - Z)
\]
and $\mathcal{A}^{q}(X \ on \ Z)$ is defined as $\pi_{-q}(\mathcal{A}(X \ on \ Z))$.

For $\mathcal{A}$ chain complexes-valued, if we write $C^{\bullet}$ to be the cone of $\mathcal{A}(X)\to \mathcal{A}(X - Z)$,
$\mathcal{A}(X \ on \ Z)$ is defined as $C^{\bullet}[-1]$ 
\[
 \mathcal{A}(X \ on \ Z) \to \mathcal{A}(X)\to \mathcal{A}(X - Z)
\]
and $\mathcal{A}^{q}(X \ on \ Z)$ is defined as $H_{-q}(\mathcal{A}(X \ on \ Z))$.
\end{definition}

Now we recall the following three definitions. 
\begin{definition}[7]
\'Etale excision

The functor $\mathcal{A}$ is said to satisfy \'etale excision if for any given diagram:
\[
 \begin{CD}
  Z  @>>> X^{'}\\
  @V=VV  @VfVV\\
  Z @>j>> X
 \end{CD}
\]
where $j: Z \to X$ is the closed immersion and $f$ is \'etale, the pullback
\[
 f^{\ast} : \mathcal{A}^{q}( X \ on \ Z) \xrightarrow{\simeq} \mathcal{A}^{q}( X^{'} \ on \ Z)
\]
is an isomorphism for any integer $q$. 
\end{definition}

\begin{definition}[7]
Zariski excision

The functor $\mathcal{A}$ is said to satisfy Zariski excision 
if the pullback
\[
 f^{\ast} : \mathcal{A}^{q}( X \ on \ Z) \xrightarrow{\simeq} \mathcal{A}^{q}( X^{'} \ on \ Z)
\]
is an isomorphism for any integer $q$, when $f$ runs over all  open immersions. 
\end{definition}

When the functor $\mathcal{A}$ satisfies Zariski excision, one can check that the above definition 3.1 does define a ``cohomology theory with support'' in the sense of [7]. The naturality can be verified by Octahedral axioms.

\begin{definition}[7]
Projective bundle formula for $\mathbb{P}^{1}$

 Let $X$ be any scheme over $k$ and $\mathbb{P}_{X}^{1}$ be the projective line over $X$. We write $\pi$ for the natural projection
\[
 \pi : \mathbb{P}_{X}^{1} \to X.
\]

The functor $\mathcal{A}$ is said to satisfy projective bundle formula for $\mathbb{P}^{1}$ if
\[
 (\pi^{\ast},O_{\mathbb{P}^{1}}(-1)\otimes \pi^{\ast}): \mathcal{A}^{q}( X )\oplus \mathcal{A}^{q}( X ) \xrightarrow{\simeq} \mathcal{A}^{q}(\mathbb{P}_{X}^{1})
\]
is an isomorphism for any integer $q$.
\end{definition}

Now, let's recall the following facts.
\begin{theorem}[5,30]
 Let $X$ be a noetherian scheme with finite dimension, then both the Thomason-Trobaugh K-theory spectrum $\mathcal{K}$ and negative cyclic homology $\mathcal{HN}$ satisfy \'etale excision and Projective bundle formula for $\mathbb{P}^{1}$.
\end{theorem}

\begin{proof}
For the Thomason-Trobaugh K-theory spectrum $\mathcal{K}$, see [30]. For negative cyclic homology $\mathcal{HN}$, see[5].
\end{proof}
\begin{theorem}[7]
 Let $X$ be a noetherian scheme with finite dimension and $\mathcal{A}$ is a functor as above. If $\mathcal{A}$ satisfies $\mathcal{A}^{q}(\emptyset) = 0$ and Zariski excision.
then there exist two strongly convergent spectral sequences:

1.Brown-Gersten spectral sequence (or Descent spectral sequence)
\[
 E_{2}^{p,q} = H_{Zar}^{p}(X,\underline{\mathcal{A}^{q}}) \Longrightarrow \mathcal{A}^{p+q}(X).
\]

2.Coniveau spectral sequence
\[
 E_{1}^{p,q} = \bigoplus_{x \in X^{p}}\mathcal{A}^{p+q}(X \ on \ x) \Longrightarrow \mathcal{A}^{p+q}(X)
\]
where $\mathcal{A}^{p+q}(X)(X \ on \ x) = \varinjlim_{x \in U}\mathcal{A}^{p+q}(U \ on \ \{x\}^{-}\cap U)$.
\end{theorem}

We explain a little bit about coniveau spectral sequence. Let
\[
 Z^{\bullet}: \ Z^{d} \subset Z^{d-1} \subset \dots \subset Z^{0}=X
\]
be a chain of closed subsets of $X$, where $codim_{X}(Z^{p})\geq p$. For a pair$(Z^{p+1} \subset Z^{p})$, the homotopy fibration 
\[
 \mathcal{A}(X \ on \ Z^{p+1}) \to \mathcal{A}(X \ on \ Z^{p}) \to \mathcal{A}(X - Z^{p+1} \ on \ Z^{p}-Z^{p+1})
\]
induces a long exact sequence:
{\footnotesize
\begin{align*}
\dots \to &\mathcal{A}^{q}(X \ on \ Z^{p+1}) \to \mathcal{A}^{q}(X \ on \ Z^{p}) \to \mathcal{A}^{q}(X - Z^{p+1} \ on \ Z^{p}-Z^{p+1}) \\
&\to \mathcal{A}^{q+1}(X \ on \ Z^{p+1}) \to \dots
\end{align*}
}
We can construct an exact couple from the above by setting 
 $D^{p,q} = \mathcal{A}^{p+q}(X \ on \ Z^{p})$ and $E^{p,q} = \mathcal{A}^{p+q}(X - Z^{p+1} \ on \ Z^{p} - Z^{p+1})$. 

Order the set of $(d+1)$-tuples $Z^{\bullet}$ by $Z^{\bullet} \leq Z^{'\bullet}$ if $Z^{p} \subset Z^{'p}$ for all $p$.
Passing to direct limit, we get a new exact couple with (a direct limit of exact couples is still an exact couple)
\[
D_{1}^{p,q} = \varinjlim _{Z^{\bullet}}\mathcal{A}^{p+q}(X \ on \ Z^{p})
\]
\[
 E_{1}^{p,q} = \varinjlim _{Z^{\bullet}}\mathcal{A}^{p+q}(X - Z^{p+1} \ on \ Z^{p} - Z^{p+1})= \bigoplus_{x \in X^{(p)}}\mathcal{A}^{p+q}(X \ on \ x)
\]
where  $X^{(p)}$ denotes the set of points of codimension $p$ in $X$
and
\[
\mathcal{A}^{p+q}(X)(X \ on \ x) = \varinjlim_{x \in U}\mathcal{A}^{p+q}(U \ on \ \{x\}^{-}\cap U).
\].

The spectral sequence associated to this new exact couple is the coniveau spectral sequence:
\[
 E_{1}^{p,q} = \bigoplus_{x \in X^{p}}\mathcal{A}^{p+q}(X \ on \ x) \Longrightarrow \mathcal{A}^{p+q}(X).
\]
Its $E_{1}$-terms give rise to Cousin complexes:
{\small
\[
 0 \to \bigoplus_{x \in X^{(0)}}\mathcal{A}^{q}(X \ on \ x) \xrightarrow{d_{1}^{0,q}} \bigoplus_{x \in X^{(1)}}\mathcal{A}^{q+1}(X \ on \ x) \xrightarrow{d_{1}^{1,q}} \bigoplus_{x \in X^{(2)}}\mathcal{A}^{q+2}(X \ on \ x) \to \dots
\]
}
We have the following augmented Gersten sequence:
\[
 0 \to D_{1}^{0,q} \to E_{1}^{0,q} \to E_{1}^{1,q} \to E_{1}^{2,q} \to \dots
\]
which means
\[
 0 \to \mathcal{A}^{q}(X) \to \bigoplus_{x \in X^{(0)}}\mathcal{A}^{q}(X \ on \ x) \to \bigoplus_{x \in X^{(1)}}\mathcal{A}^{q+1}(X \ on \ x) \to \bigoplus_{x \in X^{(2)}}\mathcal{A}^{q+2}(X \ on \ x) \to \dots
\]
We are interested in when the sheafified Gersten sequence 
\[
 0 \to \mathcal{A}^{q}(X) \to \bigoplus_{x \in X^{(0)}}\underline{\mathcal{A}}^{q}(X \ on \ x) \to \bigoplus_{x \in X^{(1)}}\underline{\mathcal{A}}^{q+1}(X \ on \ x) \to \bigoplus_{x \in X^{(2)}}\underline{\mathcal{A}}^{q+2}(X \ on \ x) \to \dots
\]
is exact, where $\bigoplus_{x \in X^{(p)}}\underline{\mathcal{A}}^{n}(X \ on \ x)$  is the sheaf associated to the presheaf, for all $n$ and $p$,
\[
 U \to \bigoplus_{x \in U^{(p)}}\mathcal{A}^{n}(U \ on \ x).
\]
In other words, $\bigoplus_{x \in X^{(p)}}\underline{\mathcal{A}}^{n}(X \ on \ x)$ is the flasque sheaf $\bigoplus_{x \in X^{(p)}}j_{x\ast}\mathcal{A}^{n}(X \ on \ x)$, where 
$j_{x}$ is the immersion $\{x\} \to X$.

The following Effacement theorem tells us when the sheafified Gersten sequence is exact.
\begin{definition}[7]
let $X$ be a scheme over $k$ and $t$ is a point in $X$.
A functor 
\[
 \mathcal{A}: \ Sch^{op}/k \to Abelian \ groups
\]
 is called effaceable at $(X,t)$, if the following conditions are satisfied:

Given any integer $p\geq 0$, a closed subvariety $Z$ of codimension at least $p+1$,
and $t \in Z$, there exists an open neighbourhood $U$of $t$, $U \subset X$, and a closed subset $Z^{'}$ containing $Z$ such that
\[
 codim_{X}(Z^{'}) \geq p
\]
and 
\[
 \mathcal{A}^{n}(U \ on \ Z\cap U) \xrightarrow{0} \mathcal{A}^{n}(U \ on \ Z^{'}\cap U)
\]
for all $n \geq 0$.
\end{definition}

Now we state the effacement theorem due to Gabber.
\begin{theorem}[7]
 Effacement theorem

Let $\mathcal{A}$ be a functor from the category $Sch^{op}/k$ , $k$ is an infinite field, to
spectra or chain complexes:
\[
 \mathcal{A} : Sch^{op}/k \to spectra
\]
or
\[
 \mathcal{A} : Sch^{op}/k \to chain \  complexes
\]

If $\mathcal{A}$ satisfies \'etale excision and projective bundle formula, then for any given $q$, $\mathcal{A}^{q}$ is effaceable at $(X,t)$,
where $X/k$ is smooth at $t$.
\end{theorem}

And we have the following corollary due to the definition of effaceable functor.
\begin{corollary}[7]
If $\mathcal{A}^{n}(X)$ is effaceable at $(X,t)$, for arbitrary $t \in X$, then 
\[
 a_{Zar}D_{1}^{p,q} \xrightarrow{i=0} a_{Zar}D_{1}^{p-1,q+1}
\]
for any $p \geq 0$, where $a_{Zar}D_{1}^{p,q}$ is the Zariski sheafification of $D_{1}^{p,q}$. Hence, the 
sheafified Gersten complexes are exact because of the following fact. 
\end{corollary}

\begin{theorem}[7]
The following statements are equivalent:

1.The sheafified Gersten sequence is exact.

2.For any $p \geq 0$ 
\[
 a_{Zar}D_{1}^{p,q} \xrightarrow{i=0} a_{Zar}D_{1}^{p-1,q+1}.
\]

3.
\[
 \varinjlim_{(Z^{p+1}\subset Z^{p})\leq (X^{p+1}\subset X^{p})}\mathcal{A}^{p+q}(X \ on \ Z^{p+1}) \xrightarrow{i=0} \varinjlim_{(Z^{p+1}\subset Z^{p})\leq (X^{p+1}\subset X^{p})}\mathcal{A}^{p+q}(X \ on \ Z^{p}).
\]
\end{theorem}

\begin{corollary}[7]
 If the sheafified Gersten sequence is exact, then the $E_{2}$-term of he coniveau spectral sequence is 
\[
 E_{2}^{p,q} = H^{p}(E^{\bullet,q},d) = H_{Zar}^{p}(X,\underline{\mathcal{A}^{q}}) 
\]
We recall that 
\[
E_{2}^{p,q} = H^{p}(E^{\bullet,q},d) = \dfrac{ker(d_{1}^{p,q})}{Im(d_{1}^{p-1,q})}. 
\]
\end{corollary}
This corollary tells us that,under the above hypothesis, the $E_{2}$ pages of coniveau spectral sequence agree with those of Brown-Gersten spectral sequence.

\begin{corollary}[7]Universal exactness

Let $\mathcal{A}$ be a spectrum valued or complexes valued functor in above. For arbitrary scheme $T/k$($T$ might be singular), we can 
define a new functor $\mathcal{A_{T}}$
\[
 X \to \mathcal{A}(X\times T).
\]
If $\mathcal{A}$ satisfies \'etale excision and projective bundle formula, then so does the new functor $\mathcal{A}(X\times T)$.
This means that if $X$ is smooth, then $\mathcal{A_{T}}$ is effaceable.
\end{corollary}

\subsection{Chern character}
\label{Chern character}

Recall that the Eilenberg- Maclane functor sends complexes to spectra. So we obtain a spectrum associated to negative cyclic homology complex. Let's still call it $\mathcal{HN}$. The following fact is pointed out in [5]. As we will see later,
our tangent maps are induced from Chern Character. 
\begin{theorem}[5]
 There exists a Chern character from $\mathcal{K}$ spectrum to $\mathcal{HN}$ spectrum, 
\[
 Ch: \mathcal{K} \to \mathcal{HN}.
\]
\end{theorem}

Now we state the main theorem of this section.
\begin{theorem}
 There exists the following commutative diagram between 
the sheafified Gersten sequences( $m$ is any integer, both sequences are flasque resolutions ):

\[
  \begin{CD}
      0 @. 0\\
      @VVV @VVV\\
      HN_{m}(O_{X}) @<Chern<< K_{m}(O_{X}) \\
      @VVV @VVV\\
      HN_{m}(k(X)) @<Chern<< K_{m}(k(X)) \\
      @VVV @VVV\\
      \oplus_{d \in X^{(1)}}\underline{HN}_{m-1}(O_{X,d} \ on \ d) @<Chern<< \oplus_{d \in X^{(1)}}\underline{K}_{m-1}(O_{X,d} \ on \ d) \\
      @VVV @VVV\\
     \oplus_{y \in X^{(2)}}\underline{HN}_{m-2}(O_{X,y} \ on \ y) @<Chern<< \oplus_{y \in X^{(2)}}\underline{K}_{m-2}(O_{X,y } \ on \ y)  \\
      @VVV @VVV\\
      \dots @<Chern<< \dots \\ 
      @VVV @VVV\\
     \oplus_{x \in X^{(n)}}\underline{HN}_{m-n}(O_{X,x} \ on \ x) @<Chern<< \oplus_{x \in X^{(n)}}\underline{K}_{m-n}(O_{X,x} \ on \ x)  \\
      @VVV @VVV\\
      0 @. 0
  \end{CD}
\]
\end{theorem}

\begin{proof}
The above Chern character at spectrum level induces maps on coniveau spectral sequences:
\[
  E_{1}^{p,q} = \bigoplus_{x \in X^{(p)}}\mathcal{K}^{p+q}(O_{X,x} \ on \ x) \longrightarrow E_{1}^{p,q} = \bigoplus_{x \in X^{(p)}}\mathcal{HN}^{p+q}(O_{X,x} \ on \ x),
\]
where 
\[
\mathcal{K}^{p+q}(O_{X,x} \ on \ x) = \pi_{-p-q}(\mathcal{K}(O_{X,x} \ on \ x))=K_{-p-q}(O_{X,x} \ on \ x)
\]
and 
\[
\mathcal{HN}^{p+q}(O_{X,x} \ on \ x) = \pi_{-p-q}(\mathcal{HN}(O_{X,x} \ on \ x))=HN_{-p-q}(O_{X,x} \ on \ x).
\]

Both the non-connective K-theory $\mathcal{K}$ and the 
negative cyclic homology $\mathcal{HN}$  satisfy \'etale excision and projective bundle formula.
Thus, $K^{q}$ and $HN^{q}$ are effaceable, $\forall q$, for any smooth scheme $X/k$.

Combining with the Effacement theorem, we obtain the desired commutative diagram of 
the sheafified Gersten sequences.

\end{proof}

\textbf{Remark}. Here, $\mathcal{HN}^{p+q}(O_{X,x} \ on \ x)(= HN_{-p-q}(O_{X,x} \ on \ x))$ is defined in terms of the direct limit
\[
\mathcal{HN}^{p+q}(O_{X,x} \ on \ x) = \varinjlim_{x \in U}\mathcal{HN}^{p+q}(U \ on \ \overline{\{x\}}\cap U).
\]
It doesn't necessarily agree with the one defined in [5, example 2.7] since negative cyclic homology doesn't commute with direct limit. Thanks M.Schlichting for pointing out this.

By the corollary 3.12, we  know that that for any integer $q$,  $K^{q}(-\times T)$ and $HN^{q}(-\times T)$ are effaceable, for any $T/k$($T$ might be singular).
Hence, we also have the following result.
\begin{theorem}
There exists the following commutative diagram between the sheafified Gersten sequences(for any integer $m$, both sequences are flasque resolutions):
\[
  \begin{CD}
      0 @. 0\\
      @VVV @VVV\\
      HN_{m}(O_{X\times T}) @<Chern<< K_{m}(O_{X\times T}) \\
      @VVV @VVV\\
      HN_{m}(k(X)\times T) @<Chern<< K_{m}(k(X)\times T) \\
      @VVV @VVV\\
      \oplus_{d\times T \in X\times T^{(1)}}\underline{HN}_{m-1}(O_{X\times T,d\times T} \ on \ d\times T) @<Chern<< \oplus_{d \times T \in X\times T^{(1)}}\underline{K}_{m-1}(O_{X\times T,d\times T} \ on \ d\times T) \\
      @VVV @VVV\\
     \oplus_{y \times T \in X\times T^{(2)}}\underline{HN}_{m-2}(O_{X\times T,y\times T} \ on \ y\times T) @<Chern<< \oplus_{y\times T \in X\times T^{(2)}}\underline{K}_{m-2}(O_{X\times T,y\times T } \ on \ y\times T)  \\
      @VVV @VVV\\
      \dots @<Chern<< \dots \\ 
      @VVV @VVV \\
      \dots @<Chern<< \dots
  \end{CD}
\]
\end{theorem}

For our purpose, we would like to take $T$ to be the dual number, ie, $spec k[\varepsilon]$.
\begin{corollary}
There exists the following commutative diagram between the exact sheafified Gersten sequences(for any integer $m$, both sequences are flasque resolutions):

\[
  \begin{CD}
      0 @. 0\\
      @VVV @VVV\\
      HN_{m}(O_{X[\varepsilon]}) @<Chern<< K_{m}(O_{X[\varepsilon]}) \\
      @VVV @VVV\\
      HN_{m}(k(X)[\varepsilon]) @<Chern<< K_{m}(k(X)[\varepsilon]) \\
      @VVV @VVV\\
      \oplus_{d[\varepsilon] \in X[\varepsilon]^{(1)}}\underline{HN}_{m-1}(O_{X,d}[\varepsilon] \ on \ d[\varepsilon]) @<Chern<< \oplus_{d[\varepsilon] \in X[\varepsilon]^{(1)}}\underline{K}_{m-1}(O_{X,d} [\varepsilon]\ on \ d[\varepsilon]) \\
      @VVV @VVV\\
     \oplus_{y[\varepsilon] \in X[\varepsilon]^{(2)}}\underline{HN}_{m-2}(O_{X,y}[\varepsilon] \ on \ y[\varepsilon]) @<Chern<< \oplus_{y[\varepsilon] \in X[\varepsilon]^{(2)}}\underline{K}_{m-2}(O_{X,y}[\varepsilon] \ on \ y[\varepsilon])  \\
      @VVV @VVV\\
      \dots @<Chern<< \dots \\ 
      @VVV @VVV\\
     \oplus_{x[\varepsilon] \in X[\varepsilon]^{(n)}}\underline{HN}_{m-n}(O_{X,x}[\varepsilon] \ on \ x[\varepsilon]) @<Chern<< \oplus_{x[\varepsilon] \in X[\varepsilon]^{(n)}}\underline{K}_{m-n}(O_{X,x}[\varepsilon] \ on \ x[\varepsilon])  \\
      @VVV @VVV\\
      0 @. 0.
  \end{CD}
\]

\end{corollary}

\section{Goodwillie and Cathelieau isomorphism}
\label{Goodwillie and Cathelieau isomorphism}
In this section, we discuss lambda and Adams operations for (negative) cyclic homology and K-theory. In 4.1, we do explicit computation on Adams' eigen-spaces of negative cyclic homology. In section 4.2, we recall lambda
and Adams operations on K-theory briefly and then extend Adams operations to negative K-groups by using a method of Weibel.
In section 4.3, we prove Goodwillie-type and Cathelineau-type results by using a recent result
of Corti$\tilde{n}$as-Haesemeyer-Weibel[6].

\subsection{Adams operations on negative cyclic homology}
\label{Adams Operations on negative cyclic homology}
In this subsection, all the variant of cyclic homology are taken over $\mathbb{Q}$. Let $A$ be any commutative $\mathbb{Q}$-algebra and $I$ be an ideal of $A$.
We can associate a  Hochschild complexes $C^{h}_{\ast}(A)$ to $A$ as in [21]. The action of the symmetric groups on $C^{h}_{\ast}(A)$
gives the lambda operation
\[
 HH_{n}(A)=HH_{n}^{(1)}(A) \oplus \dots \oplus HH_{n}^{(n)}(A),
\]
and similarly
\[
 HC_{n}(A)=HC_{n}^{(1)}(A) \oplus \dots \oplus HC_{n}^{(n)}(A),
\]
\[
 HN_{n}(A)=HN_{n}^{(1)}(A) \oplus \dots \oplus HN_{n}^{(n)}(A).
\]

 There is also a Hochschild complexes  
$C^{h}_{\ast}(A/I)$ associated to  $A/I$. We use $C^{h}_{\ast}(A,I)$ to denote the kernel of the natural map
\[
 C^{h}_{\ast}(A) \rightarrow C^{h}_{\ast}(A/I).
\]
Then the relative Hochschild module $HH_{\ast}(A,I)$ is the homology of the complex $C^{h}_{\ast}(A,I)$. Moreover, the action of the symmetric groups on $C^{h}_{\ast}(A,I)$
gives the lambda operation
\[
 HH_{n}(A,I)=HH_{n}^{(1)}(A,I) \oplus \dots \oplus HH_{n}^{(n)}(A,I)
\]
and similarly
\[
 HC_{n}(A,I)=HC_{n}^{(1)}(A,I) \oplus \dots \oplus HC_{n}^{(n)}(A,I),
\]
\[
 HN_{n}(A,I)=HN_{n}^{(1)}(A,I) \oplus \dots \oplus HN_{n}^{(n)}(A,I).
\]

Let's assume $R$ is a regular noetherian ring and also a commutative $\mathbb{Q}$-algebra from now on, and $\varepsilon$ is the dual number. We consider $R[\varepsilon]= R \oplus \varepsilon R $ as a graded $\mathbb{Q}$-algebra.
The following SBI sequence is obtained from the corresponding eigen-piece of the relative Hochschild complex:
\[
 \rightarrow HC^{(i)}_{n+1}(R[\varepsilon],\varepsilon) \xrightarrow{S} HC^{(i-1)}_{n-1}(R[\varepsilon],\varepsilon) \xrightarrow{B} HH^{(i)}_{n}(R[\varepsilon],\varepsilon) \xrightarrow{I} HC^{(i)}_{n}(R[\varepsilon],\varepsilon) \rightarrow
\]

According to a result of Geller-Weibel [10], the above S map is $0$ on $HC(R[\varepsilon],\varepsilon)$. This enable us to break the SBI sequence up into 
short exact sequence:
\[
 0 \rightarrow HC^{(i-1)}_{n-1}(R[\varepsilon],\varepsilon) \xrightarrow{B} HH^{(i)}_{n}(R[\varepsilon],\varepsilon) \xrightarrow{I} HC^{(i)}_{n}(R[\varepsilon],\varepsilon) \rightarrow 0.
\]
In the following, we will use this short exact sequence to compute $HC^{(i)}_{n}(R[\varepsilon],\varepsilon)$.
\begin{theorem}
\begin{equation}
\begin{cases}
 \begin{CD}
 HC_{n}^{(i)}(R[\varepsilon],\varepsilon)= \Omega^{{2i-n}}_{R/ \mathbb{Q}}, for \  [\frac{n}{2}] \leq i \leq n.\\
 HC_{n}^{(i)}(R[\varepsilon],\varepsilon)= 0, else.
 \end{CD}
\end{cases}
\end{equation} 

\end{theorem}

\begin{proof}
 Easy exercise which can be done  by computing $HH^{(i)}_{n}(R[\varepsilon],\varepsilon)$ and then use induction. More details can be found in[35]. 
\end{proof}
The above result tells us that
\begin{theorem}
\[
 HC_{n}(R[\varepsilon],\varepsilon) = \Omega^{{n}}_{R/ \mathbb{Q}}\oplus \Omega^{{n-2}}_{R/ \mathbb{Q}} \oplus \dots
\]
the last term is $\Omega^{{1}}_{R/ \mathbb{Q}}$ or $R$, depending on $n$ odd or even.
\end{theorem}
The following corollaries are obvious from the fact that for any commutative $k$-algebra  $A$ , where $k$ is a field of characteristic $0$, and $I$ be an ideal of $A$,
\[
 HN_{n}(A,I)=HC_{n-1}(A,I).
\]
\[
 HN_{n}^{(i)}(A,I)=HC_{n-1}^{(i-1)}(A,I).
\]

\begin{corollary}
\begin{equation}
\begin{cases}
 \begin{CD}
 HN_{n}^{(i)}(R[\varepsilon],\varepsilon)= \Omega^{{2i-n-1}}_{R/ \mathbb{Q}}, for \ [\frac{n}{2}] < i \leq n.\\
 HN_{n}^{(i)}(R[\varepsilon],\varepsilon)= 0, else.
 \end{CD}
\end{cases}
\end{equation} 
 
\end{corollary}

\begin{corollary}
\[
 HN_{n}(R[\varepsilon],\varepsilon) = \Omega^{{n-1}}_{R/ \mathbb{Q}}\oplus \Omega^{{n-3}}_{R/ \mathbb{Q}} \oplus \dots
\]
the last term is $\Omega^{{1}}_{R/ \mathbb{Q}}$ or $R$, depending on $n$ odd or even.
\end{corollary}

We can also generalize the above results to the sheaf level.
\begin{theorem}
Let $X$ be a smooth scheme over a field $k$, $chark=0$. we have the following

\begin{equation}
\begin{cases}
 \begin{CD}
 HN_{n}^{(i)}(O_{X}[\varepsilon],\varepsilon)= \Omega^{{2i-n-1}}_{O_{X}/ \mathbb{Q}}, for \  [\frac{n}{2}] < i \leq n.\\
 HN_{n}^{(i)}(O_{X}[\varepsilon],\varepsilon)= 0, else.
 \end{CD}
\end{cases}
\end{equation} 

It follows that 
\[
 HN_{n}(O_{X}[\varepsilon],\varepsilon) = \Omega^{{n-1}}_{O_{X}/ \mathbb{Q}}\oplus \Omega^{{n-3}}_{O_{X}/ \mathbb{Q}} \oplus \dots
\]
the last term is $\Omega^{{1}}_{O_{X}/ \mathbb{Q}}$ or $O_{X}$, depending on $n$ odd or even.
\end{theorem}

\subsection{Adams operations on K-groups}
\label{Adams operations on K-groups}

In this section, we recall lambda and Adams operations on K-groups briefly. We assume $X$ to be a noetherian scheme with finite Krull dimension. It is well known that the Grothendieck group of $X$ has a $\lambda$-ring structure given by exterior power,namely,$\lambda^{k}(E) = \Lambda^{k}E$ for any given vector bundle $E$ over $X$.
There are also several ways extending the Adams operations on Grothendieck groups to higher K-groups, [24] by Soul\'e, [11,12]by Gillet-Soul\'e and [13,14] by Grayson.

In [24], Soul\'e defines lambda operations on higher K-groups( with support ) and shows that there is a $\lambda$-ring structure for the higher K-groups. The key of his approach is to consider K-theory as a generalized cohomology theory:
\[
 K= \mathbb{Z} \times BGL^{+}.
\]

Hence, we have
\[
 K(X) = \mathbb{H}(X,\mathbb{Z} \times BGL^{+})
\]
and
\[
 K(X \ on \ Y) = \mathbb{H}_{Y}(X,\mathbb{Z} \times BGL^{+}).
\]
where $Y$ is closed in $X$ and and $K(X \ on \ Y)$, K-theory of $X$ with support in $Y$, is defined as the homotopy fibre of 
\[
 BQP(X) \rightarrow BQP(X \ - \ Y)
\]
here $P(X)$ is the category of locally free sheaves of finite rank on $X$ and $Q$ stands for Quillen's Q-construction.

Now, we let $R_{\mathbb{Z}}(GL_{N})$ be the Grothendieck group of representations of the general linear group scheme of $GL_{N}$. Then it is well known that $R_{\mathbb{Z}}(GL_{N})$
 has a $\lambda$-ring structure. And moreover, an element of $R_{\mathbb{Z}}(GL_{N})$ induces a morphism
\[
 \mathbb{Z} \times BGL_{N}^{+} \to \mathbb{Z} \times BGL^{+}.
\]

In other word, there is a morphism between abelian groups:
\[
 R_{\mathbb{Z}}(GL_{N}) \rightarrow [\mathbb{Z} \times BGL_{N}^{+}, \mathbb{Z} \times BGL^{+}].
\]
Passing to limit, we have
\[
 R_{\mathbb{Z}}(GL) \rightarrow [\mathbb{Z} \times BGL^{+}, \mathbb{Z} \times BGL^{+}].
\]
Furthermore, we have the following morphism by taking hypercohomology:
\[
 R_{\mathbb{Z}}(GL) \rightarrow \{\mathbb{H}_{Y}(X,\mathbb{Z} \times BGL^{+}), \mathbb{H}_{Y}(X,\mathbb{Z} \times BGL^{+})\}.
\]
And finally we arrive at group level:
\[
 R_{\mathbb{Z}}(GL) \rightarrow \{K_{m}(X \ on \ Y), K_{m}(X \ on \ Y)\}.
\]

In other word, the $\lambda$-operations on $K_{m}(X \ on \ Y)$ are induced from the $\lambda$-operations of $R_{\mathbb{Z}}(GL_{N})$. In fact, this is exact the point to prove 
$K_{m}(X \ on \ Y)$ carries a $\lambda$-ring structure.

Since the appearance of the non-zero negative K-groups in our study, we need to extend the above Adams operations  $\psi^{k}$ to negative range.
This can be done by descending induction, which was already pointed out by Weibel in[33]. 

For every integer $n \in \mathbb{Z}$,  we have the following Bass fundamental exact sequence.

\[
 ... \to K_{n}(X[t,t^{-1}] \ on \ Y[t,t^{-1}]) \to  K_{n-1}(X \ on \ Y)  \to 0.
\]

%{\tiny
%\[
 %0 \to K_{n}(X \ on \ Y) \to K_{n}(X[t] \ on \ Y[t]) \oplus K_{n}(X[t^{-1}] \ on \ Y[t^{-1}])   \to K_{n}(X[t,t^{-1}] \ on \ Y[t,t^{-1}]) \to  K_{n-1}(X \ on \ Y)  \to 0.
%\]
%}
 In particular, for any $x \in K_{-1}(X \ on \ Y)$, we have $x\cdot t \in K_{0}(X[t,t^{-1}] \ on \ Y[t,t^{-1}])$, where 
$t \in K_{1}(k[t,t^{-1}])$. We have
\[
 \psi^{k}(x\cdot t ) = \psi^{k}(x)\psi^{k}(t)= \psi^{k}(x) k\cdot t.
\]

Tensoring with $\mathbb{Q}$, we have obtained Adams operations $\psi^{k}$ on $K_{-1}(X \ on \ Y)$:
\[
 \psi^{k}(x)= \dfrac{\psi^{k}(x\cdot t )}{k\cdot t}.
\]

Continuing this procedure, we obtain Adams operations on all the negative K-groups.

\subsection{Goodwillie-type and Cathelineau-type results}
\label{Goodwillie-type and Cathelineau-type results}
In this section, we will show Goodwillie-type and Cathelineau-type results for non-connective K-groups. All the variant of cyclic homology are taken over $\mathbb{Q}$. Let's recall that in [8] Goodwillie shows the relative 
Chern character is an isomorphism between the relative K-group $K_{n}(A,I)$ and negative cyclic homology $HN_{n}(A,I)$, where $A$ is a commutative $\mathbb{Q}$-algebra and $I$
is a nilpotent ideal in $A$. 
\begin{theorem}[8] 

Let $I$ be a nilpotent ideal in a commutative $\mathbb{Q}$-algebra $A$, the relative Chern character 
\[
  Ch: K_{n}(A,I) \to HN_{n}(A,I)
\]
is an isomorphism.
\end{theorem}

%\textbf{Remark}: In fact, Goodwillie provided two isomorphisms, the relative Chern character $Ch$ and the %rational homotopy character $\rho$. Corti$\tilde{n}$as-Weibel identify the 
%relative Chern character $Ch$ with the rational homotopy character $\rho$ by showing they are induced by %maps which are naturally homotopic.

This result is further generalized by Cathelineau in [3]
\begin{theorem}[3]

 The Goodwillie's isomorphism
\[
  K_{n}(A,I) = HN_{n}(A,I)
\]
respects Adams' operations. That is,
\[
 K_{n}^{(i)}(A,I) = HN_{n}^{(i)}(A,I).
\]
\end{theorem}

In [6], Corti\~nas-Haesemeyer-Weibel show a space level version of Goodwillie's theorems in appendix B.

For every  nilpotent sheaf of ideal $I$, we define $K(O,I)$ and $HN(O,I)$ as the following presheaves respectively:
\[
 U \rightarrow K(O(U),I(U))
\]
and
\[
 U \rightarrow HN(O(U),I(U)).
\]

We write $\mathcal{K}(O,I)$ and $\mathcal{HN}(O,I)$ for the presheaves of spectrum whose initial spaces are $K(O,I)$ and $HN(O,I)$ respectively. 
Moreover, one define $\mathcal{K}^{(i)}(O,I)$ as the homotopy fiber of $\mathcal{K}(O,I)$ on which $\psi^{k}-k^{i}$ acts acyclicly. And we define $\mathcal{HN}^{(i)}(O,I)$ similarly. 
Goodwillie's theorem and Cathelineau's isomorphism can be generalized in the following way.
\begin{theorem}
Cortinas-Haesemeyer-Weibel[6]

 The relative Chern character induces homotopy equivalence of spectra:
\[
 Ch: \mathcal{K}(O,I) \simeq \mathcal{HN}(O,I)
\]
and 
\[
 Ch: \mathcal{K}^{(i)}(O,I) \simeq \mathcal{HN}^{(i)}(O,I).
\]
\end{theorem}

Now, let $X$ be a scheme essenially finite type over a field $k$, where $Char k=0$. Let $Y$ be a closed subset in a scheme $X$ and $U = X - Y$. 

Let $\mathbb{H}(X,\bullet)$ denote Thomason's hypercohomology of spectra. We have the following 
Nine-diagrams(each column and row are homotopy fibration):

\[
  \begin{CD}
     \mathbb{H}_{Y}(X, \mathcal{K}(O, \varepsilon)) @>>> \mathbb{H}(X, \mathcal{K}(O, \varepsilon)) @>>>  \mathbb{H}(U, \mathcal{K}(O, \varepsilon)) \\
     @VVV  @VVV   @VVV  \\
     \mathbb{H}_{Y}(X, \mathcal{K}(O_{X}[\varepsilon])) @>>> \mathbb{H}(X, \mathcal{K}(O_{X}[\varepsilon])) @>>>  \mathbb{H}(U, \mathcal{K}(O_{U}[\varepsilon])) \\
     @VVV  @VVV   @VVV  \\
     \mathbb{H}_{Y}(X, \mathcal{K}(O_{X})) @>>> \mathbb{H}(X, \mathcal{K}(O_{X})) @>>>  \mathbb{H}(U, \mathcal{K}(O_{U})) \\
  \end{CD}
\]
and

\[
  \begin{CD}
     \mathbb{H}_{Y}(X, \mathcal{HN}(O, \varepsilon)) @>>> \mathbb{H}(X, \mathcal{HN}(O, \varepsilon)) @>>>  \mathbb{H}(U, \mathcal{HN}(O, \varepsilon)) \\
     @VVV  @VVV   @VVV  \\
     \mathbb{H}_{Y}(X, \mathcal{HN}(O_{X}[\varepsilon])) @>>> \mathbb{H}(X, \mathcal{HN}(O_{X}[\varepsilon])) @>>>  \mathbb{H}(U, \mathcal{HN}(O_{U}[\varepsilon])) \\
     @VVV  @VVV   @VVV  \\
     \mathbb{H}_{Y}(X, \mathcal{HN}(O_{X})) @>>> \mathbb{H}(X, \mathcal{HN}(O_{X})) @>>>  \mathbb{H}(U, \mathcal{HN}(O_{U})) \\
  \end{CD}
\]

The above diagrams result in the following result
\begin{theorem} 
$\mathbb{H}_{Y}(X, \mathcal{K}(O, \varepsilon))$ is the homotpy fibre of 
\[
 \mathbb{H}_{Y}(X, \mathcal{K}(O_{X}[\varepsilon])) \rightarrow \mathbb{H}_{Y}(X, \mathcal{K}(O_{X})),
\]
and $\mathbb{H}_{Y}(X, \mathcal{HN}(O, \varepsilon))$ is the homotopy fibre of
\[
 \mathbb{H}_{Y}(X, \mathcal{HN}(O_{X}[\varepsilon])) \rightarrow \mathbb{H}_{Y}(X, \mathcal{HN}(O_{X})).
\]
\end{theorem}

Since both $\mathcal{K}$ and $\mathcal{HN}$ satisfy Zariski excision, we have the following identifications:
\[
 K(X \ on \ Y)=\mathbb{H}_{Y}(X, \mathcal{K}(O_{X})), 
\]

\[
 K(X[\varepsilon] \ on \ Y[\varepsilon])=\mathbb{H}_{Y}(X, \mathcal{K}(O_{X}[\varepsilon])), 
\]
and similiar identifications for $\mathcal{HN}$.

Combining Goodwillie's isomorphism(space version) with the above result, we have proved the following theorem, which can be considered as a Goodwillie-type isomorphism for relative 
K-groups with support.
\begin{theorem}
Let $K_{n}(X[\varepsilon] \ on \ Y[\varepsilon], \varepsilon)$ denote the kernel of 
\[
  K_{n}(X[\varepsilon] \ on \ Y[\varepsilon]) \rightarrow K_{n}(X \ on \ Y)
\]
and $HN_{n}(X[\varepsilon] \ on \ Y[\varepsilon], \varepsilon)$ denote the kernel of 
\[
  HN_{n}(X[\varepsilon] \ on \ Y[\varepsilon]) \rightarrow HN_{n}(X \ on \ Y),
\]
we have
\[
 K_{n}(X[\varepsilon] \ on \ Y[\varepsilon], \varepsilon) = HN_{n}(X[\varepsilon] \ on \ Y[\varepsilon], \varepsilon).
\]
\end{theorem}

According to [6], there exists  the following two splitting fibrations:
\[
 \mathcal{K}^{(i)}(O, \varepsilon) \rightarrow \mathcal{K}(O, \varepsilon) \rightarrow \prod_{j\neq i}\mathcal{K}^{(j)}(O, \varepsilon),
\]
and
\[
 \mathcal{HN}^{(i)}(O, \varepsilon) \rightarrow \mathcal{HN}(O, \varepsilon) \rightarrow \prod_{j\neq i}\mathcal{HN}^{(j)}(O, \varepsilon).
\]

Sine taking $\mathbb{H}_{Y}(X,-)$ perserves homotopy fibrations, there exists the following two splitting fibrations:
 \[
  \mathbb{H}_{Y}(X, \mathcal{K}^{(i)}(O, \varepsilon)) \to \mathbb{H}_{Y}(X, \mathcal{K}(O, \varepsilon))  \xrightarrow{\psi^{k}-k^{i}}    \mathbb{H}_{Y}(X ,\prod_{j\neq i}\mathcal{K}^{(j)}(O, \varepsilon)),
 \]
\[
 \mathbb{H}_{Y}(X, \mathcal{HN}^{(i)}(O, \varepsilon)) \to \mathbb{H}_{Y}(X, \mathcal{HN}(O, \varepsilon))  \xrightarrow{\psi^{k}-k^{i+1}}    \mathbb{H}_{Y}(X,\prod_{j\neq i}\mathcal{HN}^{(j)}(O, \varepsilon )).
\]

Passing to group level, we obtain the following results:
\begin{theorem}
\[
 \mathbb{H}^{-n}_{Y}(X, \mathcal{K}^{(i)}(O, \varepsilon)) = \{x \in \mathbb{H}^{-n}_{Y}(X, \mathcal{K}(O, \varepsilon))| \psi^{k}(x)-k^{i}(x)=0 \}.
\]
\[
 \mathbb{H}^{-n}_{Y}(X, \mathcal{HN}^{(i)}(O, \varepsilon)) = \{x \in \mathbb{H}^{-n}_{Y}(X, \mathcal{HN}(O, \varepsilon))| \psi^{k}(x)-k^{i+1}(x)=0 \}.
\]
\end{theorem}

We have shown that
\[
 \mathbb{H}^{-n}_{Y}(X, \mathcal{K}(O, \varepsilon)) = K_{n}(X[\varepsilon] \ on \ Y[\varepsilon], \varepsilon),
\]
and
\[
 \mathbb{H}^{-n}_{Y}(X, \mathcal{HN}(O, \varepsilon)) = HN_{n}(X[\varepsilon] \ on \ Y[\varepsilon], \varepsilon).
\]
Therefore, the homotopy equivalences
\[ 
  \mathcal{K}(O, \varepsilon) \simeq \mathcal{HN}(O, \varepsilon)
\]
and
\[
 \mathcal{K}^{(i)}(O, \varepsilon) \simeq\mathcal{HN}^{(i)}(O, \varepsilon),
\]
give us the following  refiner result:

\begin{theorem}
\[
 K_{n}^{(i)}(X[\varepsilon] \ on \ Y[\varepsilon],\varepsilon) = HN_{n}^{(i)}(X[\varepsilon] \ on \ Y[\varepsilon],\varepsilon).
\]
\end{theorem}

This result enables us to compute the relative  K-groups with support in terms of the relative negative cyclic groups with support. Now, we show an explicit computation on relative negative cyclic groups with support which will be used later.

\begin{theorem}
Suppose $X$ is a $d$-dimensional smooth projective variety over a field $k$, where $Char k=0$ and $y \in X^{(j)}$.  For any integer $m$, we have 
\[
 HN_{m}(O_{X,y}[\varepsilon] \ on \ y[\varepsilon],\varepsilon)= H_{y}^{j}(\Omega^{\bullet}_{O_{X,y}/\mathbb{Q}}),
\]
where $\Omega^{\bullet}_{O_{X,y}/\mathbb{Q}}=\Omega^{m+j-1}_{O_{X,y}/\mathbb{Q}}\oplus \Omega^{m+j-3}_{O_{X,y}/\mathbb{Q}}\oplus \dots$
\end{theorem}

\begin{proof}
$O_{X,y}$ is a regular local ring with dimension $j$, so the depth of $O_{X,y}$ is $j$. For each $n \in \mathbb{Z}$,  $\Omega^{n}_{O_{X,y}/\mathbb{Q}}$ can be written as a direct limit of 
$O_{X,y}'s$. Therefore, $\Omega^{n}_{O_{X,y}/\mathbb{Q}}$ has depth $j$.

Let's write $HN_{m}(O_{X,y}[\varepsilon] \ on \ y[\varepsilon],\varepsilon)$ for the kernel of the projection:
\[
HN_{m}(O_{X,y}[\varepsilon] \ on \ y[\varepsilon])  \xrightarrow{\varepsilon =0} HN_{m}(O_{X,y} \ on \ y).
\]
Then $HN_{m}(O_{X,y}[\varepsilon] \ on \ y[\varepsilon],\varepsilon)$ can be identified with $\mathbb{H}_{y}^{-m}(O_{X,y},HN(O_{X,y}[\varepsilon],\varepsilon))$,
where $HN(O_{X,y}[\varepsilon],\varepsilon)$ is the relative negative cyclic complex, that is the kernel of
\[
 HN(O_{X,y}[\varepsilon]) \xrightarrow{\varepsilon=0} HN(O_{X,y}).
\]

There is a spectral sequence :
\[
 H_{y}^{p}(O_{X,y}, H^{q}(HN(O_{X,y}[\varepsilon],\varepsilon))) \Longrightarrow \mathbb{H}_{y}^{-m}(HN(O_{X,y}[\varepsilon],\varepsilon)).
\]

By corollary 4.1.4, we have
 \[
 H^{q}(HN(O_{X,y}[\varepsilon],\varepsilon))= HN_{-q}(O_{X,y}[\varepsilon],\varepsilon)= \Omega^{-q-1}_{O_{X,y}/\mathbb{Q}}\oplus \Omega^{-q-3}_{O_{X,y}/\mathbb{Q}}\oplus \dots
 \]
 As each $\Omega^{n}_{O_{X,y}/\mathbb{Q}}$ has depth $j$, only $H_{y}^{j}(X,H^{q}(HN(O_{X,y}[\varepsilon],\varepsilon)))$ can survive because of the depth condition.
This means $q=-m-j$ and 
\[
 H^{-m-j}(HN(O_{X,y}[\varepsilon],\varepsilon))=HN_{m+j}(O_{X,y}[\varepsilon],\varepsilon)= \Omega^{m+j-1}_{O_{X,y}/\mathbb{Q}}\oplus \Omega^{m+j-3}_{O_{X,y}/\mathbb{Q}}\oplus \dots
\]
Let's write 
\[
\Omega^{\bullet}_{O_{X,y}/\mathbb{Q}} = \Omega^{m+j-1}_{O_{X,y}/\mathbb{Q}}\oplus \Omega^{m+j-3}_{O_{X,y}/\mathbb{Q}}\oplus \dots
\]
Thus 
\[
 \mathbb{H}_{y}^{-m}(HN(O_{X,y}[\varepsilon],\varepsilon))=H_{y}^{j}(\Omega^{\bullet}_{O_{X,y}/\mathbb{Q}}).
\]
this means
\[
 HN_{m}(O_{X,y}[\varepsilon] \ on \ y_{\varepsilon},\varepsilon)= H_{y}^{j}(\Omega^{\bullet}_{O_{X,y}/\mathbb{Q}}).
\]
\end{proof}

Repeating the above proof and noting corollary 4.1.3, we have the following finer result:
\begin{theorem}
 Suppose $X$ is a $d$-dimensional smooth projective variety over a fielf $k$, where $char k=0$ and $y \in X^{(j)}$. For any integer $m$, we have 
\[
 HN^{(i)}_{m}(O_{X,y}[\varepsilon] \ on \ y[\varepsilon],\varepsilon)= H_{y}^{j}(\Omega^{\bullet,(i)}_{O_{X,y}/\mathbb{Q}}),
\]
where 
\begin{equation}
\begin{cases}
 \begin{CD}
 \Omega_{O_{X}/ \mathbb{Q}}^{\bullet,(i)}= \Omega^{{2i-(m+j)-1}}_{O_{X}/ \mathbb{Q}}, for \  \frac{m+j}{2}  < \ i \leq m+j.\\
  \Omega_{O_{X}/ \mathbb{Q}}^{\bullet,(i)}= 0, else.
 \end{CD}
\end{cases}
\end{equation} 
\end{theorem}

Combining with theorem 4.3.5 and 4.3.7, we have the following corollary
\begin{corollary}
Under the same assumption as above, we have
\[
 K_{m}(O_{X,y}[\varepsilon] \ on \ y[\varepsilon],\varepsilon)= H_{y}^{j}(\Omega^{\bullet}_{O_{X,y}/\mathbb{Q}}),
\]
where $\Omega^{\bullet}_{O_{X,y}/\mathbb{Q}}=\Omega^{m+j-1}_{O_{X,y}/\mathbb{Q}}\oplus \Omega^{m+j-3}_{O_{X,y}/\mathbb{Q}}\oplus \dots$

Moreover, we have
\[
 K^{(i)}_{m}(O_{X,y}[\varepsilon] \ on \ y[\varepsilon],\varepsilon)= H_{y}^{j}(\Omega^{\bullet,(i)}_{O_{X,y}/\mathbb{Q}}),
\]
where 
\begin{equation}
\begin{cases}
 \begin{CD}
 \Omega_{O_{X}/ \mathbb{Q}}^{\bullet,(i)}= \Omega^{{2i-(m+j)-1}}_{O_{X}/ \mathbb{Q}}, for \  \frac{m+j}{2}  < \ i \leq m+j.\\
  \Omega_{O_{X}/ \mathbb{Q}}^{\bullet,(i)}= 0, else.
 \end{CD}
\end{cases}
\end{equation} 
\end{corollary}

\section{Main results}
\label{Main results}

The aim of this section is to prove the existence of formal tangent maps from Bloch-Gersten-Quillen sequence to Cousin resolution and also to prove the 
tangent sequence to Bloch-Gersten-Quillen sequence is Cousin resolution.  In other word,
we will see ``arrows" from the  Bloch-Gersten-Quillen sequence to the Cousin resolution in a functorial way. We shall show that the formal tangent maps can be obtained as compositions 
of the Chern character and  natural projections. 

Now suppose $X$ is a smooth projective  $n$-dimensional variety over a field $k$, $chark=0$. We state the following definition firstly which is the answer to Green-Griffiths' question in section 2:

\begin{definition}

[7]

Let $T_{j}$ denote $Spec(k[t]/(t^{j+1})$ and $X_{j}$ denote the $j$-th infinitesimal thickening, $X \times T_{j}$, the Bloch-Quillen-Gersten sequence $\mathcal{G}_j$ on $X_{j}$ is defined to be the following flasque resolution($m$ can be any integer):
{\footnotesize
\begin{align*}
0 \to & K_{m}(O_{X_{j}}) \to K_{m}(k(X)_{j}) \to \bigoplus_{x_{j} \in X_{j} ^{(1)}}\underline{K}_{m-1}(O_{X_{j},x_{j}} \ on \ x_{j}) \to \dots \\
 & \dots \to \bigoplus_{x_{j} \in X_{j} ^{(n)}}\underline{K}_{m-n}(O_{X_{j},x_{j}} \ on \ x_{j}) \to 0.
\end{align*}
}
where $O_{X_{j}}=O_{X\times T_{j}}$,  $k(X)_{j}= k(X)\times T_{j}$, $x_{j}=x \times T_{j}$. $\underline{K}_{p}(O_{X_{j},x_{j}} \ on \ x_{j})$ is the flasque sheaf $j_{x_{j}\ast}K_{p}(O_{X_{j},x_{j}} \ on \ x_{j})$, where 
$j_{x_{j}}$ is the immersion $\{x_{j}\} \to X$.
\end{definition}

\begin{proof}
The existence of the sequence $\mathcal{G}_j$ follows from theorem 3.6 on page 8. The fact of falsque resolution follows from corollary 3.12.
\end{proof}

We will show
\begin{theorem}
 There exists the following commutative diagram($m$ can be any integer, each column is a flasque resolution, $Pr_{i}$ 's are natural projections):
{\scriptsize
\[
  \begin{CD}
     0 @. 0 @. 0\\
     @VVV @VVV @VVV\\
     \Omega_{O_{X}/ \mathbb{Q}}^{\bullet} @<Pr_{1}<< HN_{m}(O_{X[\varepsilon]}) @<Chern<< K_{m}(O_{X}[\varepsilon]) \\
     @VVV @VVV @VVV\\
     \Omega_{k(X)/ \mathbb{Q}}^{\bullet} @<Pr_{2}<<  HN_{m}(k(X)[\varepsilon]) @<Chern<< K_{m}(k(X)[\varepsilon]) \\
     @VVV @VVV @VVV\\
     \oplus_{d \in X^{(1)}}\underline{H}_{d}^{1}(\Omega_{O_{X}/\mathbb{Q}}^{\bullet}) @<Pr_{3}<< \oplus_{d[\varepsilon]\in X[\varepsilon]^{(1)}}\underline{HN}_{m-1}(O_{X,d}[\varepsilon] \ on \ d[\varepsilon]) @<Chern<<  \oplus_{d[\varepsilon]\in X[\varepsilon]^{(1)}}\underline{K}_{m-1}(O_{X,d}[\varepsilon] \ on \ d[\varepsilon])\\
     @VVV @VVV @VVV\\
     \oplus_{y \in X^{(2)}}\underline{H}_{y}^{2}(\Omega_{O_{X}/ \mathbb{Q}}^{\bullet}) @<Pr_{4}<< \oplus_{y[\varepsilon] \in X[\varepsilon]^{(2)}}\underline{HN}_{m-2}(O_{X,y}[\varepsilon] \ on \ y[\varepsilon]) @<Chern<< \oplus_{y[\varepsilon] \in X[\varepsilon]^{(2)}}\underline{K}_{m-2}(O_{X,y}[\varepsilon] \ on \ y[\varepsilon]) \\
     @VVV @VVV @VVV\\
      \dots @<Pr<< \dots @<Chern<< \dots \\ 
     @VVV @VVV @VVV\\
     \oplus_{x\in X^{(n)}}\underline{H}_{x}^{n}(\Omega_{O_{X}/ \mathbb{Q}}^{\bullet}) @<Pr_{n+2}<< \oplus_{x[\varepsilon]\in X[\varepsilon]^{(n)}}\underline{HN}_{m-n}(O_{X,x}[\varepsilon] \ on \ x[\varepsilon]) @<Chern<<  \oplus_{x[\varepsilon]\in X[\varepsilon]^{(n)}}\underline{K}_{m-n}(O_{X,x}[\varepsilon] \ on \ x[\varepsilon]) \\
     @VVV @VVV @VVV\\
      0 @. 0 @. 0
  \end{CD}
\]
}
where 
\begin{equation}
\begin{cases}
 \begin{CD}
 \Omega_{O_{X}/ \mathbb{Q}}^{\bullet} = \Omega^{m-1}_{O_{X}/\mathbb{Q}}\oplus \Omega^{m-3}_{O_{X}/\mathbb{Q}}\oplus \dots\\
  \Omega_{k(X)/ \mathbb{Q}}^{\bullet} = \Omega^{m-1}_{k(X)/\mathbb{Q}}\oplus \Omega^{m-3}_{k(X)/\mathbb{Q}}\oplus \dots
 \end{CD}
\end{cases}
\end{equation}

\end{theorem}

We will show the following lemma first.
\begin{lemma}
 There exists the following commutative splitting diagram($m$ can be any integer, each column is a flasque resolution):

{\footnotesize
\[
  \begin{CD}
     0 @. 0 @. 0\\
     @VVV @VVV @VVV\\
     \Omega_{O_{X}/ \mathbb{Q}}^{\bullet} @<Pr_{1}<< HN_{m}(O_{X[\varepsilon]}) @<<< HN_{m}(O_{X}) \\
     @VVV @VVV @VVV\\
     \Omega_{k(X)/ \mathbb{Q}}^{\bullet} @<Pr_{2}<<  HN_{m}(k(X)[\varepsilon]) @<<< HN_{m}(k(X)) \\
     @VVV @VVV @VVV\\
     \oplus_{d \in X^{(1)}}\underline{H}_{d}^{1}(\Omega_{O_{X}/\mathbb{Q}}^{\bullet}) @<Pr_{3}<< \oplus_{d[\varepsilon]\in X[\varepsilon]^{(1)}}\underline{HN}_{m-1}(O_{X,d}[\varepsilon] \ on \ d[\varepsilon]) @<<<  \oplus_{d \in X^{(1)}}\underline{HN}_{m-1}(O_{X,d} \ on \ d)\\
     @VVV @VVV @VVV\\
     \oplus_{y \in X^{(2)}}\underline{H}_{y}^{2}(\Omega_{O_{X}/ \mathbb{Q}}^{\bullet}) @<Pr_{4}<< \oplus_{y[\varepsilon] \in X[\varepsilon]^{(2)}}\underline{HN}_{m-2}(O_{X,y}[\varepsilon] \ on \ y[\varepsilon]) @<<< \oplus_{y \in X^{(2)}}\underline{HN}_{m-2}(O_{X,y} \ on \ y) \\
     @VVV @VVV @VVV\\
      \dots @<Pr<< \dots @<<< \dots \\ 
     @VVV @VVV @VVV\\
     \oplus_{x\in X^{(n)}}\underline{H}_{x}^{n}(\Omega_{O_{X}/ \mathbb{Q}}^{\bullet}) @<Pr_{n+2}<< \oplus_{x[\varepsilon]\in X[\varepsilon]^{(n)}}\underline{HN}_{m-n}(O_{X,x}[\varepsilon] \ on \ x[\varepsilon]) @<<<  \oplus_{x \in X^{(n)}}\underline{HN}_{m-n}(O_{X,x} \ on \ x) \\
     @VVV @VVV @VVV\\
      0 @. 0 @. 0
  \end{CD}
\]
}
where 
\begin{equation}
\begin{cases}
 \begin{CD}
 \Omega_{O_{X}/ \mathbb{Q}}^{\bullet} = \Omega^{m-1}_{O_{X}/\mathbb{Q}}\oplus \Omega^{m-3}_{O_{X}/\mathbb{Q}}\oplus \dots\\
  \Omega_{k(X)/ \mathbb{Q}}^{\bullet} = \Omega^{m-1}_{k(X)/\mathbb{Q}}\oplus \Omega^{m-3}_{k(X)/\mathbb{Q}}\oplus \dots
 \end{CD}
\end{cases}
\end{equation}
\end{lemma}

\begin{proof}
Since $HN_{p}(O_{X,x}[\varepsilon] \ on \ x[\varepsilon], \varepsilon)= HC_{p-1}(O_{X,x}[\varepsilon] \ on \ x[\varepsilon], \varepsilon)$, the latter relative group agrees with that of [5, example 2.7].  
So one can use theorem 4.13. Thanks for Schlichting for pointing out this.
\end{proof}

Now we give a proof of theorem 5.2.
\begin{proof}
(Proof of theorem 5.2)

The above result tells us there are natural projections from Gersten sequence involving negative cyclic homology to Cousin resolution of $\Omega_{O_{X}/ \mathbb{Q}}^{\bullet}$:
{\small
\[
  \begin{CD}
     0 @. 0 \\
     @VVV @VVV \\
     \Omega_{O_{X}/ \mathbb{Q}}^{\bullet} @<Pr_{1}<< HN_{m}(O_{X[\varepsilon]})  \\
     @VVV @VVV \\
     \Omega_{k(X)/ \mathbb{Q}}^{\bullet} @<Pr_{2}<<  HN_{m}(k(X)[\varepsilon]) \\
     @VVV @VVV \\
     \oplus_{d \in X^{(1)}}\underline{H}_{d}^{1}(\Omega_{O_{X}/\mathbb{Q}}^{\bullet}) @<Pr_{3}<< \oplus_{d[\varepsilon]\in X[\varepsilon]^{(1)}}\underline{HN}_{m-1}(O_{X,d}[\varepsilon] \ on \ d[\varepsilon])\\
     @VVV @VVV \\
     \oplus_{y \in X^{(2)}}\underline{H}_{y}^{2}(\Omega_{O_{X}/ \mathbb{Q}}^{\bullet}) @<Pr_{4}<< \oplus_{y[\varepsilon] \in X[\varepsilon]^{(2)}}\underline{HN}_{m-2}(O_{X,y}[\varepsilon] \ on \ y[\varepsilon]) \\
     @VVV @VVV \\
      \dots @<Pr<< \dots  \\ 
     @VVV @VVV \\
     \oplus_{x\in X^{(n)}}\underline{H}_{x}^{n}(\Omega_{O_{X}/ \mathbb{Q}}^{\bullet}) @<Pr_{n+2}<< \oplus_{x[\varepsilon]\in X[\varepsilon]^{(n)}}\underline{HN}_{m-n}(O_{X,x}[\varepsilon] \ on \ x[\varepsilon])  \\
     @VVV @VVV \\
      0 @. 0 
  \end{CD}
\]
}

We have shown the following commutative diagram induced by Chern character in corollary 3.16:

\[
  \begin{CD}
     0 @. 0 \\
     @VVV @VVV \\
      HN_{m}(O_{X[\varepsilon]}) @<Chern<< K_{m}(O_{X}[\varepsilon]) \\
     @VVV @VVV\\
     HN_{m}(k(X)[\varepsilon]) @<Chern<< K_{m}(k(X)[\varepsilon]) \\
     @VVV @VVV\\
     \oplus_{d[\varepsilon]\in X[\varepsilon]^{(1)}}\underline{HN}_{m-1}(O_{X,d}[\varepsilon] \ on \ d[\varepsilon]) @<Chern<<  \oplus_{d[\varepsilon]\in X[\varepsilon]^{(1)}}\underline{K}_{m-1}(O_{X,d}[\varepsilon] \ on \ d[\varepsilon])\\
     @VVV @VVV\\
      \oplus_{y[\varepsilon] \in X[\varepsilon]^{(2)}}\underline{HN}_{m-2}(O_{X,y}[\varepsilon] \ on \ y[\varepsilon]) @<Chern<< \oplus_{y[\varepsilon] \in X[\varepsilon]^{(2)}}\underline{K}_{m-2}(O_{X,y}[\varepsilon] \ on \ y[\varepsilon]) \\
    @VVV @VVV\\
      \dots @<Chern<< \dots \\ 
    @VVV @VVV\\
      \oplus_{x[\varepsilon]\in X[\varepsilon]^{(n)}}\underline{HN}_{m-n}(O_{X,x}[\varepsilon] \ on \ x[\varepsilon]) @<Chern<<  \oplus_{x[\varepsilon]\in X[\varepsilon]^{(n)}}\underline{K}_{m-n}(O_{X,x}[\varepsilon] \ on \ x[\varepsilon]) \\
      @VVV @VVV\\
      0 @. 0
  \end{CD}
\]

Combining the above two commutative diagrams, we see there exists the following commutative diagrams.
{\scriptsize
\[
  \begin{CD}
     0 @. 0 @. 0\\
     @VVV @VVV @VVV\\
     \Omega_{O_{X}/ \mathbb{Q}}^{\bullet} @<Pr_{1}<< HN_{m}(O_{X[\varepsilon]}) @<Chern<< K_{m}(O_{X}[\varepsilon]) \\
     @VVV @VVV @VVV\\
     \Omega_{k(X)/ \mathbb{Q}}^{\bullet} @<Pr_{2}<<  HN_{m}(k(X)[\varepsilon]) @<Chern<< K_{m}(k(X)[\varepsilon]) \\
     @VVV @VVV @VVV\\
     \oplus_{d \in X^{(1)}}\underline{H}_{d}^{1}(\Omega_{O_{X}/\mathbb{Q}}^{\bullet}) @<Pr_{3}<< \oplus_{d[\varepsilon]\in X[\varepsilon]^{(1)}}\underline{HN}_{m-1}(O_{X,d}[\varepsilon] \ on \ d[\varepsilon]) @<Chern<<  \oplus_{d[\varepsilon]\in X[\varepsilon]^{(1)}}\underline{K}_{m-1}(O_{X,d}[\varepsilon] \ on \ d[\varepsilon])\\
     @VVV @VVV @VVV\\
     \oplus_{y \in X^{(2)}}\underline{H}_{y}^{2}(\Omega_{O_{X}/ \mathbb{Q}}^{\bullet}) @<Pr_{4}<< \oplus_{y[\varepsilon] \in X[\varepsilon]^{(2)}}\underline{HN}_{m-2}(O_{X,y}[\varepsilon] \ on \ y[\varepsilon]) @<Chern<< \oplus_{y[\varepsilon] \in X[\varepsilon]^{(2)}}\underline{K}_{m-2}(O_{X,y}[\varepsilon] \ on \ y[\varepsilon]) \\
     @VVV @VVV @VVV\\
      \dots @<Pr<< \dots @<Chern<< \dots \\ 
     @VVV @VVV @VVV\\
     \oplus_{x\in X^{(n)}}\underline{H}_{x}^{n}(\Omega_{O_{X}/ \mathbb{Q}}^{\bullet}) @<Pr_{n+2}<< \oplus_{x[\varepsilon]\in X[\varepsilon]^{(n)}}\underline{HN}_{m-n}(O_{X,x}[\varepsilon] \ on \ x[\varepsilon]) @<Chern<<  \oplus_{x[\varepsilon]\in X[\varepsilon]^{(n)}}\underline{K}_{m-n}(O_{X,x}[\varepsilon] \ on \ x[\varepsilon]) \\
     @VVV @VVV @VVV\\
      0 @. 0 @. 0
  \end{CD}
\]
}
\end{proof}

\begin{definition}
 The formal tangent maps from the  Bloch-Gersten-Quillen 
sequence to the Cousin resolution are defined as compositions of Chern character and natural projections as above.
\end{definition}

\begin{corollary}
There exists formal tangent maps from the  Bloch-Gersten-Quillen sequence to the Cousin resolution:
\[
  \begin{CD}
     0 @.  0\\
     @VVV  @VVV\\
     \Omega_{O_{X}/ \mathbb{Q}}^{\bullet}  @<tan1<< K_{m}(O_{X}[\varepsilon]) \\
     @VVV @VVV \\
     \Omega_{k(X)/ \mathbb{Q}}^{\bullet}  @<tan2<< K_{m}(k(X)[\varepsilon]) \\
     @VVV @VVV \\
     \oplus_{d \in X^{(1)}}\underline{H}_{d}^{1}(\Omega_{O_{X}/\mathbb{Q}}^{\bullet})  @<tan3<<  \oplus_{d[\varepsilon]\in X[\varepsilon]^{(1)}}\underline{K}_{m-1}(O_{X,d}[\varepsilon] \ on \ d[\varepsilon])\\
     @VVV @VVV \\
     \oplus_{y \in X^{(2)}}\underline{H}_{y}^{2}(\Omega_{O_{X}/ \mathbb{Q}}^{\bullet})  @<tan4<< \oplus_{y[\varepsilon] \in X[\varepsilon]^{(2)}}\underline{K}_{m-2}(O_{X,y}[\varepsilon] \ on \ y[\varepsilon]) \\
     @VVV @VVV \\
      \dots @<tan<< \dots \\ 
     @VVV @VVV\\
     \oplus_{x\in X^{(n)}}\underline{H}_{x}^{n}(\Omega_{O_{X}/ \mathbb{Q}}^{\bullet})  @<tan(n+2)<< \oplus_{x[\varepsilon]\in X[\varepsilon]^{(n)}}\underline{K}_{m-n}(O_{X,x}[\varepsilon] \ on \ x[\varepsilon]) \\
     @VVV @VVV \\
      0 @. 0 
  \end{CD}
\]
where $tani$ is defined as $Pr_{i}\circ Ch$.
\end{corollary}

Combining the above diagram with results  on computing relative K-groups with support,theorem 4.15, we get the following theorem which says that 
the formal tangent sequence to the Bloch-Gersten-Quillen 
sequence is the Cousin resolution.
\begin{theorem}
 The formal tangent sequence to the Bloch-Gersten-Quillen 
sequence is the Cousin resolution. That is there exists the following splitting commutative diagram:
{\footnotesize
\[
  \begin{CD}
     0 @. 0 @. 0\\
     @VVV @VVV @VVV\\
     \Omega_{O_{X}/ \mathbb{Q}}^{\bullet} @<tan1<< K_{m}(O_{X[\varepsilon]}) @<<< K_{m}(O_{X}) \\
     @VVV @VVV @VVV\\
     \Omega_{k(X)/ \mathbb{Q}}^{\bullet} @<tan2<<  K_{m}(k(X)[\varepsilon]) @<<< K_{m}(k(X)) \\
     @VVV @VVV @VVV\\
     \oplus_{d \in X^{(1)}}\underline{H}_{d}^{1}(\Omega_{O_{X}/\mathbb{Q}}^{\bullet}) @<tan3<< \oplus_{d[\varepsilon]\in X[\varepsilon]^{(1)}}\underline{K}_{m-1}(O_{X,d}[\varepsilon] \ on \ d[\varepsilon]) @<<<  \oplus_{d \in X^{(1)}}\underline{K}_{m-1}(O_{X,d} \ on \ d)\\
     @VVV @VVV @VVV\\
     \oplus_{y \in X^{(2)}}\underline{H}_{y}^{2}(\Omega_{O_{X}/ \mathbb{Q}}^{\bullet}) @<tan4<< \oplus_{y[\varepsilon] \in X[\varepsilon]^{(2)}}\underline{K}_{m-2}(O_{X,y}[\varepsilon] \ on \ y[\varepsilon]) @<<< \oplus_{y \in X^{(2)}}\underline{K}_{m-2}(O_{X,y} \ on \ y) \\
     @VVV @VVV @VVV\\
      \dots @<tan<< \dots @<<< \dots \\ 
     @VVV @VVV @VVV\\
     \oplus_{x\in X^{(n)}}\underline{H}_{x}^{n}(\Omega_{O_{X}/ \mathbb{Q}}^{\bullet}) @<tan(n+2)<< \oplus_{x[\varepsilon]\in X[\varepsilon]^{(n)}}\underline{K}_{m-n}(O_{X,x}[\varepsilon] \ on \ x[\varepsilon]) @<<<  \oplus_{x \in X^{(n)}}\underline{K}_{m-n}(O_{X,x} \ on \ x) \\
     @VVV @VVV @VVV\\
      0 @. 0 @. 0
  \end{CD}
\]
}
where $tani$ is defined above.
\end{theorem}

Based on theorem 4.15, we also have the following commutative diagram which roughly says Adams operations $\psi^{k}$ on K-theory can decompose the above diagram into eigen-components. We note that negative K-groups might appear. One can extend Adams' operations to negative range by using Weibel's method, recalled in section4.2. Let $K^{(i)}_{p}(O_{X,x} \ on \ x)$ denote the eigen-space of $\psi^{k}=k^{i}$.

\begin{theorem}
 There exists the following splitting commutative diagram:
{\footnotesize
\[
  \begin{CD}
     0 @. 0 @. 0\\
     @VVV @VVV @VVV\\
     \Omega_{O_{X}/ \mathbb{Q}}^{\bullet,(i)} @<tan1<< K^{(i)}_{m}(O_{X[\varepsilon]}) @<<< K^{(i)}_{m}(O_{X}) \\
     @VVV @VVV @VVV\\
     \Omega_{k(X)/ \mathbb{Q}}^{\bullet,(i)} @<tan2<<  K^{(i)}_{m}(k(X)[\varepsilon]) @<<< K^{(i)}_{m}(k(X)) \\
     @VVV @VVV @VVV\\
     \oplus_{d \in X^{(1)}}\underline{H}_{d}^{1}(\Omega_{O_{X}/\mathbb{Q}}^{\bullet,(i)}) @<tan3<< \oplus_{d[\varepsilon]\in X[\varepsilon]^{(1)}}\underline{K}^{(i)}_{m-1}(O_{X,d}[\varepsilon] \ on \ d[\varepsilon]) @<<<  \oplus_{d \in X^{(1)}}\underline{K}^{(i)}_{m-1}(O_{X,d} \ on \ d)\\
     @VVV @VVV @VVV\\
     \oplus_{y \in X^{(2)}}\underline{H}_{y}^{2}(\Omega_{O_{X}/ \mathbb{Q}}^{\bullet,(i)}) @<tan4<< \oplus_{y[\varepsilon] \in X[\varepsilon]^{(2)}}\underline{K}^{(i)}_{m-2}(O_{X,y}[\varepsilon] \ on \ y[\varepsilon]) @<<< \oplus_{y \in X^{(2)}}\underline{K}^{(i)}_{m-2}(O_{X,y} \ on \ y) \\
     @VVV @VVV @VVV\\
      \dots @<tan<< \dots @<<< \dots \\ 
     @VVV @VVV @VVV\\
     \oplus_{x\in X^{(n)}}\underline{H}_{x}^{n}(\Omega_{O_{X}/ \mathbb{Q}}^{\bullet,(i)}) @<tan(n+2)<< \oplus_{x[\varepsilon]\in X[\varepsilon]^{(n)}}\underline{K}^{(i)}_{m-n}(O_{X,x}[\varepsilon] \ on \ x[\varepsilon]) @<<<  \oplus_{x \in X^{(n)}}\underline{K}^{(i)}_{m-n}(O_{X,x} \ on \ x) \\
     @VVV @VVV @VVV\\
      0 @. 0 @. 0
  \end{CD}
\]
}
where 
\begin{equation}
\begin{cases}
 \begin{CD}
 \Omega_{O_{X}/ \mathbb{Q}}^{\bullet,(i)}= \Omega^{{2i-m+1}}_{O_{X}/ \mathbb{Q}}, for \  \frac{m-1}{2}  < \ i \leq m-1.\\
  \Omega_{O_{X}/ \mathbb{Q}}^{\bullet,(i)}= 0, else.
 \end{CD}
\end{cases}
\end{equation} 
\end{theorem}

Moreover, we have the following fact by using similiar methods:

\begin{theorem}
There exists the following splitting commutative diagram($m$ and $j$ can be any integer, $X_{j}$ is the j-th infinitesimal thickening, each column is a flasque resolution):
{\footnotesize
\[
  \begin{CD}
     0 @. 0 @. 0\\
     @VVV @VVV @VVV\\
     (\Omega_{O_{X}/ \mathbb{Q}}^{\bullet})^{\oplus j} @<tan1<< K_{m}(O_{X_{j}}) @<<< K_{m}(O_{X}) \\
     @VVV @VVV @VVV\\
     (\Omega_{k(X)/ \mathbb{Q}}^{\bullet})^{\oplus j} @<tan2<<  K_{m}(k(X)_{j}) @<<< K_{m}(k(X)) \\
     @VVV @VVV @VVV\\
     \oplus_{d \in X^{(1)}}\underline{H}_{d}^{1}((\Omega_{O_{X}/\mathbb{Q}}^{\bullet})^{\oplus j} @<tan3<< \oplus_{d_{j} \in X_{j} ^{(1)}}\underline{K}_{m-1}(O_{X_{j},d_{j}} \ on \ d_{j}) @<<<  \oplus_{d \in X^{(1)}}\underline{K}_{m-1}(O_{X,d} \ on \ d)\\
     @VVV @VVV @VVV\\
     \oplus_{y \in X^{(2)}}\underline{H}_{y}^{2}((\Omega_{O_{X}/ \mathbb{Q}}^{\bullet})^{\oplus j} @<tan4<< \oplus_{y_{j} \in X_{j} ^{(2)}}\underline{K}_{m-2}(O_{X_{j},y_{j}} \ on \ y_{j}) @<<< \oplus_{y \in X^{(2)}}\underline{K}_{m-2}(O_{X,y} \ on \ y) \\
     @VVV @VVV @VVV\\
      \dots @<tan<< \dots @<<< \dots \\ 
     @VVV @VVV @VVV\\
     \oplus_{x\in X^{(n)}}\underline{H}_{x}^{n}((\Omega_{O_{X}/ \mathbb{Q}}^{\bullet})^{\oplus j} @<tan(n+2)<< \oplus_{x_{j} \in X_{j} ^{(n)}}\underline{K}_{m-n}(O_{X_{j},x_{j}} \ on \ x_{j}) @<<<  \oplus_{x \in X^{(n)}}\underline{K}_{m-n}(O_{X,x} \ on \ x) \\
     @VVV @VVV @VVV\\
      0 @. 0 @. 0
  \end{CD}
\]
}
where
\begin{equation}
\begin{cases}
 \begin{CD}
 \Omega_{O_{X}/ \mathbb{Q}}^{\bullet} = \Omega^{m-1}_{O_{X}/\mathbb{Q}}\oplus \Omega^{m-3}_{O_{X}/\mathbb{Q}}\oplus \dots\\
  \Omega_{k(X)/ \mathbb{Q}}^{\bullet} = \Omega^{m-1}_{k(X)/\mathbb{Q}}\oplus \Omega^{m-3}_{k(X)/\mathbb{Q}}\oplus \dots
 \end{CD}
\end{cases}
\end{equation} 
 
\end{theorem}

\end{document}